\documentclass[12pt,reqno]{amsart}
\usepackage{amsmath, amsthm, amssymb, amsfonts}

\advance \topmargin by -\headheight
\advance \topmargin by -\headsep
     
\evensidemargin \oddsidemargin
\newtheorem{theorem}{Theorem}[section]
\newtheorem{lemma}[theorem]{Lemma}
\newtheorem{corollary}[theorem]{Corollary}

\theoremstyle{definition}

\theoremstyle{remark}

\numberwithin{equation}{section}

\newcommand{\mmod}[1]{\,\,(\text{mod}\,\,#1)}

  \def\bff{{\mathbf f}}

\def\bfm{{\mathbf m}}
\def\bfn{{\mathbf n}}

\def\bfu{{\mathbf u}}
\def\bfv{{\mathbf v}}
\def\bfw{{\mathbf w}}
\def\bfx{{\mathbf x}}
\def\bfy{{\mathbf y}}
\def\bfz{{\mathbf z}}

\def\calB{{\mathcal B}} 
\def\calC{{\mathcal C}} 

\def\calD{{\mathcal D}}

 \def\Ktil{{\widetilde K}}

\def\calN{{\mathcal N}}

\def\calP{{\mathcal P}}
\def\calQ{{\mathcal Q}}
\def\calR{{\mathcal R}}

\def\Gtil{\widetilde G}\def\Itil{\widetilde I}\def\Ktil{\widetilde K}

\def\dbC{{\mathbb C}}\def\dbN{{\mathbb N}}
\def\dbR{{\mathbb R}}
\def\dbZ{{\mathbb Z}}

\def\grf{{\mathfrak f}}\def\grF{{\mathfrak F}}
\def\grG{{\mathfrak G}}
\def\grH{{\mathfrak H}} \def\grI{{\mathfrak I}}

\def\grm{{\mathfrak m}}
\def\grS{{\mathfrak S}}

\def\alp{{\alpha}} \def\bfalp{{\boldsymbol \alpha}}
\def\bet{{\beta}}  
\def\gam{{\gamma}} \def\Gam{{\Gamma}}
\def\del{{\delta}} \def\Del{{\Delta}}
\def\zet{{\zeta}} \def\bfzet{{\boldsymbol \zeta}} 
\def\bfeta{{\boldsymbol \eta}} 
\def\tet{{\theta}} \def\bftet{{\boldsymbol \theta}} \def\Tet{{\Theta}}
\def\kap{{\kappa}}
\def\lam{{\lambda}}  

\def\bfxi{{\boldsymbol \xi}}

\def\sig{{\sigma}}  

\def\Ups{{\Upsilon}} 
\def\bfpsi{{\boldsymbol \psi}}
\def\ome{{\omega}} \def\Ome{{\Omega}}
\def\d{{\partial}}
\def\eps{\varepsilon}

\def\le{\leqslant} \def\ge{\geqslant}

\def\d{{\, d}}

\def\llbracket{\lbrack\;\!\!\lbrack} \def\rrbracket{\rbrack\;\!\!\rbrack}

\makeatletter
\renewcommand{\pmod}[1]{\allowbreak\mkern7mu({\operator@font mod}\,\,#1)}
\makeatother
\newcommand{\bal}{\[\begin{aligned}}
\newcommand{\eal}{\end{aligned}\]}
\newcommand{\be}{\begin{equation}}
\newcommand{\ee}{\end{equation}}
\newcommand{\ssum}[1]{\sum_{\substack{#1}}}  

\renewcommand{\(}{\left(}
\renewcommand{\)}{\right)}

\usepackage{xspace}
\newcommand{\RH}{H\"older's inequality\xspace}  
\newcommand{\er}{{\rm e}} 

\begin{document}
\title[Vinogradov's mean value theorem]{On Vinogradov's mean value theorem:\\ strongly diagonal
 behaviour\\ via efficient congruencing}
\author[Kevin Ford]{Kevin Ford$^\dag$}
\address{Department of Mathematics, University of Illinois at Urbana-Champaign, 1409 West Green St.,
 Urbana, IL 61801, USA}
\email{ford@math.uiuc.edu}
\author[Trevor D. Wooley]{Trevor D. Wooley$^*$}
\address{School of Mathematics, University of Bristol, University Walk, Clifton, Bristol BS8 1TW, United
 Kingdom}
\email{matdw@bristol.ac.uk}
\thanks{$^\dag$Supported in part by National Science Foundation grants DMS-0901339 and
 DMS-1201442.}
\thanks{$^*$Supported in part by a Royal Society Wolfson Research Merit Award.}
\date{18 December 2013}

\begin{abstract} We enhance the efficient congruencing method for estimating Vinogradov's integral for
 moments of order $2s$, with $1\le s\le k^2-1$. In this way, we prove the main conjecture for such even
 moments when $1\le s\le \tfrac{1}{4}(k+1)^2$, showing that the moments exhibit strongly diagonal
 behaviour in this range. There are improvements also for larger values of $s$, these finding application to
 the asymptotic formula in Waring's problem.
\end{abstract}
\maketitle

\section{Introduction} Considerable progress has recently been achieved in the theory of Vinogradov's
 mean value theorem (see \cite{Woo2012a}, \cite{Woo2013}), associated estimates finding
 application throughout analytic number theory, in Waring's problem and the theory of the Riemann zeta
 function, to name but two. The vehicle for these advances is the so-called ``efficient congruencing''
 method, the most striking consequence of which is that the main conjecture in Vinogradov's mean value
 theorem holds with a number of variables only twice the number conjectured to be best possible (see
 \cite[Theorem 1.1]{Woo2012a}). Our goal in the present paper is to establish the main conjecture in the
 complementary variable regime, showing that diagonal behaviour dominates for half of the range
 conjectured. In common with the previous work cited, this work far exceeds in this direction the
 conclusions available hitherto for any Diophantine system of large degree $k$.\par

When $k$ and $s$ are natural numbers, denote by $J_{s,k}(X)$ the number of integral solutions of the
 system of Diophantine equations
\be\label{1.1}
\sum_{i=1}^s (x_i^j-y_i^j)=0 \quad (1\le j\le k),
\ee
with $1\le x_i,y_i\le X$ $(1\le i\le s)$. The lower bound
\be\label{1.2}
 J_{s,k}(X) \gg X^s + X^{2s-\frac12 k(k+1)},
\ee
arises by considering the diagonal solutions of the system (\ref{1.1}) with $x_i=y_i$ $(1\le i\le s)$,
 together with a lower bound for the product of local densities (see \cite[equation (7.5)]{Vau1997}).
 Motivated by the latter considerations, the {\it main conjecture} in Vinogradov's mean value theorem
 asserts that for each $\eps>0$, one has\footnote{Throughout this paper, the implicit constant in
 Vinogradov's notation $\ll$ and $\gg$ may depend on $s$, $k$ and $\eps$.}
\be\label{1.3}
J_{s,k}(X) \ll X^\eps (X^s + X^{2s-\frac12 k(k+1)}).
\ee
In \S7 of this paper, we prove the main conjecture \eqref{1.3} for $1\le s\le \frac14 (k+1)^2$.

\begin{theorem}\label{theorem1.1}
 Suppose that $k\ge 4$ and $1\le s\le \frac14 (k+1)^2$. Then for each $\eps>0$, one has
\begin{equation}\label{1.4}
J_{s,k}(X) \ll X^{s+\eps}.
\end{equation}
\end{theorem}

In the range $1\le s\le k$, the upper bound $J_{s,k}(X)\ll X^s$ follows directly from the
Vi\'ete-Girard-Newton formulae
 concerning the roots of polynomials. Hitherto, the only other case in which the bound (\ref{1.4}) had
 been established was that in which $s=k+1$ (see \cite[Lemma 5.4]{Hua1965}, and \cite{VW1997} for a
 sharper variant). The extension of the range $1\le s\le k+1$, in which the bound (\ref{1.4}) is known to
 hold, to $1\le s\le \tfrac{1}{4}(k+1)^2$ covers half of the total range predicted by the main conjecture.
 Previous approximations to strongly diagonal behaviour in the range $1\le s\le \tfrac{1}{4}(k+1)^2$
 were considerably weaker. The second author established that when $s\le k^{3/2}(\log k)^{-1}$, one has
 the bound
$$J_{s,k}(X)\ll X^{s+\nu_{s,k}+\eps},$$
with $\nu_{s,k}=\exp(-Ak^3/s^2)$, for a certain positive constant $A$ (see \cite{Woo1994}), and with
 $\nu_{s,k}=4s/k^2$ in the longer range $s\le \tfrac{1}{4}(k+1)^2$ (see \cite{Woo2013}). Both results
 improve on earlier work of Arkhipov and Karatsuba \cite{AK1978} and Tyrina \cite{Tyr1987}, these
 authors offering substantially sharper bounds than the classical work of Vinogradov \cite{Vin1947} for
 smaller values of $s$.\par

We also improve upon bounds for $J_{s,k}(X)$ given in \cite{Woo2012a} and \cite{Woo2013} in the
 range $\tfrac{1}{4}(k+1)^2<s<k^2-1$.

\begin{theorem}\label{theorem1.2}One has the following upper bounds for $J_{s,k}(X)$.
\item{(i)} Let $s$ and $m$ be non-negative integers with
$$2m\le k\quad \text{and}\quad s\ge (k-m)^2+(k-m).$$
Then for each $\eps>0$, one has
\begin{equation}\label{1.5}
J_{s,k}(X)\ll X^{2s-\frac{1}{2}k(k+1)+\del_{k,m}+\eps},
\end{equation}
where
$$\del_{k,m}=m^2.$$
\item{(ii)} Let $s$ and $m$ be non-negative integers with
$$2m\le k-1\quad \text{and}\quad s\ge (k-m)^2-1.$$
Then for each $\eps>0$, one has the upper bound (\ref{1.5}) with
$$\del_{k,m}=m^2+m+\frac{m}{k-m-1}.$$
\end{theorem}

We note that the second bound of Theorem \ref{theorem1.2}, with $m=0$, recovers Theorem 1.1 of
 \cite{Woo2013}, which asserts that the bound \eqref{1.3} holds for $s\ge k^2-1$. Meanwhile, the first
 bound of Theorem \ref{theorem1.2}, again with $m=0$, recovers the earlier estimate provided by the
 main theorem of \cite{Woo2012a}, which delivered (\ref{1.3}) for $s\ge k^2+k$.\par

One measure of the strength of Theorem \ref{theorem1.2} compared with previous work is provided by
 the bound for $J_{s,k}(X)$ furnished in the central case $s=\frac12 k(k+1)$. For this value of $s$, it
 follows from \cite[Theorem 1.4]{Woo2013} that
$$J_{s,k}(X)\ll X^{s+\Del},$$
with $\Del=\frac{1}{8}k^2+O(k)$. Meanwhile, Theorem 1.2 above establishes such a bound with 
$\Del=(\tfrac{3}{2}-\sqrt{2})k^2+O(k)$. Note that
$$\tfrac{3}{2}-\sqrt{2}=0.085786\ldots <0.125=\tfrac{1}{8}.$$
More generally, in the situation with $s=\alp k^2$, in which $\alp$ is a parameter with $\tfrac{1}{4}\le
 \alp \le 1$, we find from \cite[Theorem 1.4]{Woo2013} that
$$J_{s,k}(X) \ll X^{2s-\frac12 k(k+1)+\Del(\alp)},$$
where $\Del(\alp)=\frac12 (1-\alp)^2k^2+O(k)$. Theorem \ref{theorem1.2}, on the other hand, shows
 that such a bound holds with $\Del(\alp)=(1-\sqrt{\alp})^2k^2+O(k)$. Note on this occasion that when
 $\tfrac{1}{4}\le \alp<1$ one has
$$(1-\sqrt{\alp})^2<\tfrac{1}{2}(1-\alp)^2,$$
as is easily verified by a modest computation.\par

Theorems \ref{theorem1.1} and \ref{theorem1.2} are special cases of a more general estimate, and it is
 the proof of this which is our focus in \S\S2 to 7.

\begin{theorem}\label{theorem1.3}
Suppose that $k$, $r$ and $t$ are positive integers with
\be\label{1.6}
k\ge 2,\quad \max\{2, \tfrac{1}{2}(k-1)\} \le t \le k,\quad 1\le r\le k\quad \text{and}\quad r+t\ge k.
\ee
Define $\kap=\kap(r,t,k)$ by
\begin{equation}\label{1.7}
\kap=r(t+1)-\tfrac{1}{2}(t+r-k)\left( t+r-k-1+\frac{2r-2}{t-1}\right).
\end{equation}
Then for each $\eps>0$, one has
$$J_{r(t+1),k}(X)\ll X^{2r(t+1)-\kap+\eps}.$$
\end{theorem}

Theorem \ref{theorem1.1} follows directly from Theorem \ref{theorem1.3} on taking $r$ and $t$ to be
 suitable integers satisfying $r+t=k$. When $k$ is even we put $r=t=k/2$, and when $k$ is odd we
 instead put $r=\tfrac{1}{2}(k+1)$ and $t=\tfrac{1}{2}(k-1)$. In each case it follows that $s=r(t+1)$ is
 the largest integer not exceeding $\tfrac{1}{4}(k+1)^2$, and we have $J_{s,k}(X)\ll X^{s+\eps}$. For
 smaller values of $s$, the same conclusion is a consequence of the convexity of exponents that follows 
from \RH\footnote{\RH was evidently first proved, in a form different from that usually found in textbooks,
 by L. J. Rogers, {\it An extension of a certain theorem in inequalities}, Messenger of Math., New Series
 {\bf XVII} (10) (February 1888), 145--150.}.\par

Theorem \ref{theorem1.2} follows in the first case from Theorem \ref{theorem1.3} on putting 
$r=t=k-m$, since then we obtain
\begin{align*}
\kap(r,t,k)&=(k-m)(k-m+1)-\tfrac{1}{2}(k-2m)(k-2m+1)\\
&=\tfrac{1}{2}k(k+1)-m^2.
\end{align*}
Meanwhile, in the second case we put $r=k-m-1$ and $t=k-m$, in this instance obtaining
\begin{align*}
\kap(r,t,k)&=(k-m-1)(k-m+1)-\tfrac{1}{2}(k-2m-1)\left( k-2m-\frac{2}{k-m-1}\right) \\
&=\tfrac{1}{2}k(k+1)-m^2-m-\frac{m}{k-m-1}.
\end{align*}

In broad strokes, Theorem \ref{theorem1.3} is obtained by fully incorporating the ideas of Arkhipov and
 Karatsuba \cite{AK1978} and Tyrina \cite{Tyr1987} into the efficient congruencing method which was first
 created in \cite{Woo2012a} and further developed in \cite{Woo2013}. The parameters $r$ and $t$
 control the way in which solutions of certain systems of congruences are counted (see \eqref{3.1} below).
 The power of the method is enhanced by the flexibility to choose the latter parameters, constrained only
 by \eqref{1.6}. In particular, the work in \cite{Woo2012a} corresponds to the case $r=t=k$, while
 \cite{Woo2013} covers the cases $t=k$ and $r+t=k+1$. We describe in more detail the role played by
 $r$ and $t$ in \S 3. The reader will find the fundamental estimate which lies at the core of our argument
 in Lemma \ref{3.3} below.\par

There are consequences of the new estimates supplied by Theorem \ref{theorem1.2} in particular so far as
 the asymptotic formula in Waring's problem is concerned. By applying the mean value estimates published
 in work \cite{For1995} of the first author in combination with mean value estimates restricted to minor
 arcs established in work \cite{Woo2012b} of the second author, one may convert improved estimates in
 Vinogradov's mean value theorem into useful estimates for mean values of exponential sums over $k$th
 powers. These in turn lead to improvements in bounds for the number of variables required to establish
 the anticipated asymptotic formula in Waring's problem. In the present paper we enhance these tools by
 engineering a hybrid of these approaches, increasing further the improvements stemming from Theorem
 \ref{theorem1.2}. We discuss this new hybrid approach in \S8, exploring in \S9 consequences for the
 asymptotic formula in Waring's problem. The details are somewhat complicated, and so we refer the
 reader to the latter section for a summary of the bounds now available.\par

The authors are grateful to Xiaomei Zhao for identifying an oversight in the original proof of Lemma \ref{lemma7.2} that we have remedied in the argument described in the present paper.
The authors also thank the referee for carefully reading the paper and for a number of useful comments.

\section{Preliminaries}

We initiate the proof of Theorem \ref{theorem1.3} by setting up the apparatus necessary for the
 application of the efficient congruencing method. Here, we take the opportunity to introduce a number of
 simplifications over the treatments of \cite{Woo2012a} and \cite{Woo2013} that have become apparent as
 the method has become more familiar. Since we consider the integer $k$ to be fixed, we abbreviate
 $J_{s,k}(X)$ to $J_s(X)$ without further comment. Our attention is focused on bounding $J_s(X)$ where,
 for the moment, we think of $s$ as being an arbitrary natural number. We define the real number
 $\lam_s^*$ by means of the relation
$$\lam_s^*=\underset{X\rightarrow \infty}{\lim \sup}\, \frac{\log J_s(X)}{\log X}.$$
It follows that, for each $\del>0$, and any real number $X$ sufficiently large in terms of $s$, $k$ and
 $\del$, one has $J_s(X)\ll X^{\lam_s^*+\del}$. In the language of \cite{Woo2012a} and
 \cite{Woo2013}, the real number $\lam_s^*$ is the infimum of the set of exponents $\lam_s$
 permissible for $s$ and $k$. In view of the lower bound (\ref{1.2}), together with a trivial bound for
 $J_s(X)$, we have
\be\label{2.1}
\max\{s,2s-\tfrac{1}{2}k(k+1)\}\le \lam_s^*\le 2s,
\ee
while the conjectured upper bound (\ref{1.3}) implies that the first inequality in \eqref{2.1} should hold
 with equality.\par

Next, we record some conventions that ease our expositary burden in what follows. The letters $k$, 
$r$ and $t$ denote fixed positive integers satisfying \eqref{1.6}, and $$s=rt.$$ We make sweeping use of 
vector notation. In particular, we may write $\bfz\equiv \bfw\mmod{p}$ to denote that 
$z_i\equiv w_i\mmod{p}$ $(1\le i\le r)$, $\bfz\equiv \xi\mmod{p}$ to denote that $z_i\equiv
\xi\mmod{p}$ $(1\le i\le r)$, or $[\bfz \mmod{q}]$ to denote the $r$-tuple $(\zet_1,\ldots,\zet_r)$,
 where for $1\le i\le r$ one has $1\le \zet_i\le q$ and $z_i \equiv \zet_i\pmod{q}$. Also, we employ the
 convention that whenever $G:[0,1)^k\rightarrow \dbC$ is integrable, then
$$\oint G(\bfalp)\d\bfalp =\int_{[0,1)^k}G(\bfalp)\d\bfalp.$$

\par For brevity, we write $\lam=\lam_{s+r}^*$. Our goal is to show that $\lam \le 2(s+r)-\kappa$, in
 which $\kap$ is the carefully chosen target exponent given in (\ref{1.7}). Let $N$ be an arbitrary natural
 number, sufficiently large in terms of $s$, $k$, $t$ and $r$, and put
\begin{equation}\label{2.2}
\tet=(16t)^{-N-1}\quad \text{and}\quad \del =(1000Nt^N)^{-1}\tet .
\end{equation}
In view of the definition of $\lam$, there exists a sequence of natural numbers $(X_\ell)_{\ell=1}^\infty$,
 tending to infinity, with the property that
\begin{equation}\label{2.3}
J_{s+r}(X_\ell)>X_\ell^{\lam-\del}\quad (\ell\in \dbN).
\end{equation}
Also, provided that $X_\ell$ is sufficiently large, one has the corresponding upper bound
\begin{equation}\label{2.4}
J_{s+r}(Y)<Y^{\lam+\del}\quad \text{for}\quad Y \ge X_{\ell}^{1/2}.
\end{equation}
In the argument that follows, we take a fixed element $X=X_\ell$ of the sequence
 $(X_\ell)_{\ell=1}^\infty$, which we may assume to be sufficiently large in terms of $s$, $k$, $r$, $t$
 and $N$. We then put $M=X^\tet$. Throughout, constants implied in the notation of Landau and
 Vinogradov may depend on $s$, $k$, $r$, $t$, $N$, $\tet$, and $\del$, but not on any other variable.\par

Let $p$ be a fixed prime number with $M<p\le 2M$ to be chosen in due course. That such a prime exists
 is a consequence of the Prime Number Theorem. When $c$ and $\xi$ are non-negative integers, and
 $\bfalp \in [0,1)^k$, define
\begin{equation}\label{2.5}
\grf_c(\bfalp;\xi)=\sum_{\substack{1\le x\le X\\ x\equiv \xi\mmod{p^c}}}\er(\alp_1x+\alp_2x^2+\ldots
 +\alp_kx^k),
\end{equation}
where $\er(z)$ denotes the imaginary exponential $\er^{2\pi i z}$. As in \cite{Woo2013}, we must consider well-conditioned 
$r$-tuples of integers belonging to distinct congruence classes modulo a suitable power of $p$. The
 following notations are similar to, though slightly simpler than, the corresponding notations introduced in
 \cite{Woo2012a} and \cite{Woo2013}. Denote by $\Xi_c^r(\xi)$ the set of $r$-tuples 
$(\xi_1,\ldots ,\xi_r)$, with
$$1\le \xi_i\le p^{c+1}\quad \text{and}\quad \xi_i\equiv \xi\pmod{p^c}\quad (1\le i\le r),$$
and such that $\xi_1,\ldots,\xi_r$ are distinct modulo $p^{c+1}$. We then define
\begin{equation}\label{2.6}
\grF_c(\bfalp;\xi)=\sum_{\bfxi\in \Xi_c^r(\xi)}\prod_{i=1}^r\grf_{c+1}(\bfalp;\xi_i),
\end{equation}
where the exponential sums $\grf_{c+1}(\bfalp;\xi_i)$ are defined via (\ref{2.5}).\par

Two mixed mean values play leading roles within our arguments. When $a$ and $b$ are positive integers,
 we define
\begin{equation}\label{2.7}
I_{a,b}(X;\xi,\eta)=\oint |\grF_a(\bfalp;\xi)^2\grf_b(\bfalp;\eta)^{2s}|\d\bfalp
\end{equation}
and
\begin{equation}\label{2.8}
K_{a,b}(X;\xi,\eta)=\oint |\grF_a(\bfalp;\xi)^2\grF_b (\bfalp;\eta)^{2t}|\d \bfalp .
\end{equation}
For future reference, we note that as a consequence of orthogonality, the mean value $I_{a,b} (X;\xi,\eta)$
 counts the number of integral solutions of the system
\begin{equation}\label{2.9}
\sum_{i=1}^r (x_i^j-y_i^j)=\sum_{l=1}^s(v_l^j-w_l^j)\quad (1\le j\le k),
\end{equation}
with
$$1\le \bfx,\bfy,\bfv,\bfw\le X,\quad \bfv\equiv\bfw\equiv \eta\pmod{p^b},$$
$$[\bfx \mmod{p^{a+1}}] \in \Xi_a^r(\xi)\quad \text{and}\quad [\bfy \mmod{p^{a+1}}]
\in \Xi_a^r(\xi).$$
Similarly, the mean value $K_{a,b} (X;\xi,\eta)$ counts the number of integral solutions of the system
\begin{equation}\label{2.10}
\sum_{i=1}^r (x_i^j-y_i^j)=\sum_{l=1}^t\sum_{m=1}^r (v_{lm}^j-w_{lm}^j)\quad (1\le j\le k),
\end{equation}
with
$$1\le \bfx,\bfy\le X,\quad [\bfx \mmod{p^{a+1}}]\in \Xi_a^r(\xi),\quad [\bfy \mmod{p^{a+1}}]
\in \Xi_a^r(\xi),$$
and for $1\le l\le t$,   
$$1\le \bfv_l,\bfw_l\le X,\quad [\bfv_l\mmod{p^{b+1}}]\in \Xi_b^r(\eta),\quad
 [\bfw_l\mmod{p^{b+1}}]\in \Xi_b^r(\eta).$$

\par It is convenient to put
\begin{equation}\label{2.11}
I_{a,b}(X)=\max_{1\le \xi\le p^a}\max_{\substack{1\le \eta\le p^b\\ \eta\not\equiv
 \xi\mmod{p}}}I_{a,b}(X;\xi,\eta)
\end{equation}
and
\begin{equation}\label{2.12}
K_{a,b}(X)=\max_{1\le \xi\le p^a}\max_{\substack{1\le \eta\le p^b\\ \eta\not\equiv \xi\mmod{p}}} 
K_{a,b}(X;\xi,\eta).
\end{equation}
Of course, these mean values implicitly depend on our choice of $p$, and this will depend on $s$, $k$,
 $r$, $t$, $\tet$ and $X_\ell$ alone. Since we fix $p$ in the pre-congruencing step described in \S6,
 following the proof of Lemma \ref{lemma6.1}, the particular choice will be rendered irrelevant.\par

The pre-congruencing step requires a definition of $K_{0,b}(X)$ consistent with the conditioning idea, and
 this we now describe. When $\xi$ is an integer and $\bfzet$ is a tuple of integers, we denote by
 $\Xi(\bfzet)$ the set of $r$-tuples $(\xi_1,\ldots ,\xi_r) \in \Xi_0^r(0)$ such that 
$\xi_i\not \equiv \zet_j\mmod{p}$ for all $i$ and $j$. Recalling \eqref{2.5}, we put
\begin{equation}\label{2.13}
\grF(\bfalp;\bfzet)=\sum_{\bfxi\in \Xi(\bfzet)}\prod_{i=1}^r\grf_1(\bfalp;\xi_i).
\end{equation}
Finally, we define
\begin{align}
\Itil_c(X;\eta)&=\oint |\grF(\bfalp;\eta)^2\grf_c(\bfalp;\eta)^{2s}|\d\bfalp ,\label{2.14}\\
\Ktil_c(X;\eta)&=\oint |\grF(\bfalp;\eta)^2\grF_c (\bfalp;\eta)^{2t}|\d\bfalp ,\label{2.15}\\
K_{0,c}(X)&=\max_{1\le \eta\le p^c} \Ktil_c (X;\eta).\label{2.16}
\end{align}

\par As in \cite{Woo2012a} and \cite{Woo2013}, our arguments are simplified by making transparent the
 relationship between mean values and their anticipated magnitudes. In this context, we define 
$\llbracket J_{s+r}(X)\rrbracket$ by means of the relation
\begin{equation}\label{2.17}
J_{s+r}(X)=X^{2s+2r-\kap}\llbracket J_{s+r}(X)\rrbracket .
\end{equation}
 Also, we define $\llbracket I_{a,b}(X)\rrbracket$ and $\llbracket K_{a,b}(X)\rrbracket$ by means of the
 relations
\begin{equation}\label{2.18}
I_{a,b}(X)=(X/M^b)^{2s}(X/M^a)^{2r-\kap}\llbracket I_{a,b}(X)\rrbracket 
\end{equation}
and
\begin{equation}\label{2.19}
K_{a,b}(X)=(X/M^b)^{2s}(X/M^a)^{2r-\kap}\llbracket K_{a,b}(X)\rrbracket .
\end{equation}
The lower bound \eqref{2.3}, in particular, may now be written as
\newcommand{\zero}{\Lambda}
\begin{equation}\label{2.20}
\llbracket J_{s+r}(X)\rrbracket >X^{\zero-\del},
\end{equation}
where we have written
\be\label{2.21}
\zero =\lam-2(s+r)+\kappa.
\ee

\par We finish this section by recalling a simple estimate from \cite{Woo2012a} that encapsulates the
 translation-dilation invariance of the Diophantine system \eqref{1.1}.

\begin{lemma}\label{lemma2.1}
Suppose that $c$ is a non-negative integer with $c\tet\le 1$. Then for each natural number $u$, one has
$$\max_{1\le \xi\le p^c}\oint |\grf_c(\bfalp;\xi)|^{2u}\d\bfalp \ll_u J_u(X/M^c).$$
\end{lemma}

\begin{proof} This is \cite[Lemma 3.1]{Woo2012a}.\end{proof}

We record an immediate consequence of Lemma \ref{lemma2.1} useful in what follows. 

\begin{corollary}\label{corollary2.2}
 Suppose that $c$ and $d$ are non-negative integers with $c\le \tet^{-1}$ and $d\le \tet^{-1}$. Then
 whenever $u,v\in \dbN$ and $\xi,\zet\in\dbZ$, one has
$$\oint \left| \grf_c(\bfalp;\xi)^{2u}\grf_d(\bfalp;\zeta)^{2v}\right|\d\bfalp \ll_{u,v} 
\(J_{u+v}(X/M^c)\)^{u/(u+v)}\( J_{u+v}(X/M^d) \)^{v/(u+v)}.$$
\end{corollary}

\begin{proof} This follows at once from Lemma \ref{lemma2.1} via \RH.\end{proof}

%
\section{Auxiliary systems of congruences}
%

Following the pattern established in \cite{Woo2012a}, in which efficient congruencing was introduced, and
 further developed in \cite{Woo2013}, we begin the main thrust of our analysis with a discussion of the
 congruences that play a critical role in our method. 
\par

Recall the conditions \eqref{1.6} on $k$, $r$ and $t$. When $a$ and $b$ are integers with $1\le a<b$, we
 denote by $\calB_{a,b}^r(\bfm;\xi,\eta)$ the set of solutions of the system of congruences
\begin{equation}\label{3.1}
\sum_{i=1}^r (z_i-\eta)^j\equiv m_j\mmod{p^{jb}}\quad (1\le j\le k),
\end{equation}
with $1\le \bfz\le p^{kb}$ and $\bfz\equiv \bfxi\pmod{p^{a+1}}$ for some $\bfxi\in \Xi_a^r(\xi)$. We
 define an equivalence relation $\calR(\lam)$ on integral $r$-tuples by declaring the $r$-tuples $\bfx$ and
 $\bfy$ to be $\calR(\lam)$-equivalent when $\bfx\equiv\bfy\mmod{p^\lam}$. We then write
 $\calC_{a,b}^{r,t}(\bfm;\xi,\eta)$ for the set of $\calR(tb)$-equivalence classes of 
$\calB_{a,b}^r(\bfm;\xi,\eta)$, and we define $B_{a,b}^{r,t}(p)$ by putting
\begin{equation}\label{3.2}
B_{a,b}^{r,t}(p)=\max_{1\le \xi\le p^a}\max_{\substack{1\le \eta\le p^b\\ \eta\not\equiv
 \xi\mmod{p}}}\max_{1\le \bfm\le p^{kb}}\text{card}(\calC_{a,b}^{r,t}(\bfm;\xi,\eta)).
\end{equation}

\par When $a=0$ we modify these definitions, so that $\calB_{0,b}^r(\bfm;\xi,\eta)$ denotes the set of
 solutions of the system of congruences (\ref{3.1}) with $1\le \bfz\le p^{kb}$ and $\bfz\equiv
 \bfxi\mmod{p}$ for some $\bfxi\in \Xi_0^r(\xi)$, and for which in addition one has $z_i\not\equiv
 \eta\mmod{p}$ for $1\le i\le r$. As in the previous case, we write $\calC_{0,b}^{r,t}(\bfm;\xi,\eta)$ for
 the set of $\calR(tb)$-equivalence classes of $\calB_{0,b}^r(\bfm;\xi,\eta)$, but we define 
$B_{0,b}^{r,t}(p)$ by putting
\begin{equation}\label{3.3}
B_{0,b}^{r,t}(p)=\max_{1\le \eta\le p^b}\max_{1\le \bfm\le p^{kb}}\text{card}
(\calC_{0,b}^{r,t}(\bfm;0,\eta)).
\end{equation}
We note that although the choice of $\xi$ in this situation with $a=0$ is irrelevant, it is notationally
 convenient to preserve the similarity with the situation in which $a\ge 1$.\par

Our argument exploits the non-singularity of the solution set underlying $B_{a,b}^{r,t}(p)$ by means of a
 version of Hensel's lemma made available within the following lemma.

\begin{lemma}\label{lemma3.1}
Let $f_1,\ldots ,f_d$ be polynomials in $\dbZ[x_1,\ldots ,x_d]$ with respective degrees $k_1,\ldots ,k_d$,
 and write
$$J(\bff;\bfx)=\mathrm{det}\left( \frac{\partial f_j}{\partial x_i}(\bfx)\right)_{1\le i,j\le d}.$$
When $\varpi$ is a prime number, and $l$ is a natural number, let $\calN(\bff;\varpi^l)$ denote the
 number of solutions of the simultaneous congruences
$$f_j(x_1,\ldots ,x_d)\equiv 0\pmod{\varpi^l}\quad (1\le j\le d),$$
with $1\le x_i\le \varpi^l$ $(1\le i\le d)$ and $(J(\bff;\bfx),\varpi)=1$. Then 
$\calN(\bff;\varpi^l)\le k_1\cdots k_d$.
\end{lemma}

\begin{proof} This is \cite[Theorem 1]{Woo1996}.\end{proof}

We recall also an auxiliary lemma from \cite{Woo2013}, 
in which terms are eliminated between related polynomial expansions.

\begin{lemma}\label{lemma3.2} Let $\alp$ and $\bet$ be natural numbers. Then there exist integers
 $c_l$ $(\alp\le l\le \alp+\bet)$ and $d_m$ $(\bet\le m\le \alp+\bet)$, depending at most on $\alp$ and
 $\bet$, and with $d_\bet\ne 0$, for which one has the polynomial identity
$$c_\alp+\sum_{l=1}^\bet c_{\alp+l}(x+1)^{\alp+l}=\sum_{m=\bet}^{\alp+\bet}d_mx^m.$$
\end{lemma}

\begin{proof} This is \cite[Lemma 3.2]{Woo2013}.\end{proof}

Our approach to bounding $B_{a,b}^{r,t}(p)$ proceeds by discarding the $k-r$ congruences of smallest
modulus $p^{jb}$ $(1\le j\le k-r)$, but nonetheless aims to lift all solutions to the modulus $p^{tb}$.
 The idea of reducing the lifting required, which is tantamount to taking $t<k$, was first exploited  by
 Arkhipov and Karatsuba \cite{AK1978} in the setting of Linnik's classical $p$-adic approach
 \cite{Lin1943}. Likewise, taking $r<k$ removes from consideration those congruences that require the
 greatest lifting and produce the biggest inefficiency in the method.  Tyrina \cite{Tyr1987} took 
$r=t\ge k/2$ and further improved bounds on $J_{s,k}(X)$ for $s=O(k^2)$.  Later, the second author
used a hybrid approach (see \cite[Lemma 2.1]{Woo1994}), with $r$ and $t$ as free parameters, to obtain
large improvements to the bounds for $s=O(k^{3/2-\eps})$.

We also follow a very general approach here, keeping $r$ and $t$ as free parameters, subject only to the
 necessary constraints given in \eqref{1.6}. For Theorem \ref{theorem1.1}, the crucial observation is that 
when $r+t=k$, then there is no lifting at all and we capture only diagonal solutions in the symmetric
 version of \eqref{3.1}. This observation is reflected in the fact that the coefficients $\mu$ and $\nu$
 imminently to be defined satisfy the condition $\mu=\nu=0$ in this situation.\par

The following lemma generalises Lemmata 3.3 to 3.6 of \cite{Woo2013}. For future reference, at this point
 we introduce the coefficients
\begin{equation}\label{3.4}
\mu=\tfrac{1}{2}(t+r-k)(t+r-k-1)\quad \text{and}\quad \nu=\tfrac{1}{2}(t+r-k)(k+r-t-1).
\end{equation}

\begin{lemma}\label{lemma3.3} Suppose that $k$, $r$ and $t$ satisfy the conditions \eqref{1.6}, and
 further that $a$ and $b$ are integers with $0\le a<b$ and $b\ge (k-t-1)a$. Then
$$B_{a,b}^{r,t}(p)\le k!p^{\mu b+\nu a}.$$
\end{lemma}

\begin{proof} We suppose in the first instance that $a\ge 1$. Fix integers $\xi$ and $\eta$ with
$$1\le \xi\le p^a,\quad 1\le \eta\le p^b\quad \text{and}\quad \eta\not\equiv \xi\mmod{p}.$$
We consider the set of $\calR(tb)$-equivalence classes of solutions $\calC_{a,b}^{r,t}(\bfm;\xi,\eta)$ of
 the system (\ref{3.1}), in our first step upgrading a subset of the congruences to the same level. Put 
$$\rho=k-r+1\quad \text{and}\quad \ome=\max \{0,k-t-1\}.$$
We denote by $\calD_1(\bfn)$ the set of $\calR(tb)$-equivalence classes of solutions of the system of
 congruences
\begin{equation}\label{3.5}
\sum_{i=1}^r(z_i-\eta)^j\equiv n_j\mmod{p^{tb+\ome a}}\quad (\rho\le j\le k),
\end{equation}
with $1\le \bfz\le p^{kb}$ and $\bfz\equiv \bfxi\mmod{p^{a+1}}$ for some 
$\bfxi\in \Xi_a^r(\xi)$. 

\par Recall our assumed bound $b\ge \ome a$ and fix an integral $k$-tuple $\bfm$. To any solution 
$\bfz$ of \eqref{3.1} there corresponds a unique $r$-tuple $\bfn=(n_\rho,\ldots,n_k)$ with 
$1\le \bfn\le p^{tb+\ome a}$ for which \eqref{3.5} holds and
$$n_j\equiv m_j\mmod{p^{\sigma(j)}}\quad (\rho\le j\le k),$$
where $\sig(j)=\min \{jb,tb+\ome a\}$. We therefore infer that
\[
\calC_{a,b}^{r,t}(\bfm;\xi,\eta) \subseteq \bigcup_{\substack{1\le n_\rho\le p^{tb+\ome a}\\ n_\rho\equiv 
m_\rho\mmod{p^{\sig(\rho)}}}}\cdots \bigcup_{\substack{1\le n_k\le p^{tb+\ome a}\\ n_k\equiv
 m_k\mmod{p^{\sig(k)}}}} \calD_1(\bfn).
\]
The number of $r$-tuples $\bfn$ in the union is equal to
$$\prod_{j=\rho}^t p^{(t-j)b+\ome a}=(p^b)^{\frac{1}{2}(t-\rho)(t-\rho+1)} (p^a)^{(t-\rho+1)
\ome }=p^{\mu b+(t-\rho+1)\ome a}.$$
Consequently, 
\begin{equation}\label{3.6}
\text{card}(\calC_{a,b}^{r,t}(\bfm;\xi,\eta))\le p^{\mu b+(t-\rho+1)\ome a}
\max_{1\le \bfn\le p^{tb+\ome a}}\text{card}(\calD_1(\bfn)).
\end{equation}

\par
Observe that for any solution $\bfz'$ of \eqref{3.5} there is an $\calR(tb)$-equivalent solution
$\bfz$ satisfying $1\le \bfz \le p^{tb+\omega a}$.
We next rewrite each variable $z_i$ in the shape $z_i=p^ay_i+\xi$. In view of the hypothesis that 
$\bfz\equiv \bfxi\mmod{p^{a+1}}$ for some $\bfxi\in \Xi_a^r(\xi)$, the $r$-tuple $\bfy$ necessarily
 satisfies 
\begin{equation}\label{3.7}
y_i\not\equiv y_m\mmod{p}\quad (1\le i<m\le r).
\end{equation}
Write $\zet=\xi-\eta$, and note that the constraint $\eta\not\equiv \xi\mmod{p}$ ensures that 
$p\nmid \zet$. We denote the multiplicative inverse of $\zet$ modulo $p^{tb+\ome a}$ by $\zet^{-1}$.
 In this way we deduce from (\ref{3.5}) that
 $\text{card}(\calD_1(\bfn))$ is bounded above by the number of $\calR(tb-a)$-equivalence classes of
 solutions of the system of congruences
\begin{equation}\label{3.8}
\sum_{i=1}^r (p^ay_i\zet^{-1}+1)^j\equiv n_j(\zet^{-1})^j\mmod{p^{tb+\ome a}}\quad 
(\rho\le j\le k),
\end{equation}
with $1\le \bfy\le p^{tb+(\ome-1)a}$ satisfying (\ref{3.7}). Let $\bfy=\bfw$ be any solution of the
 system (\ref{3.8}), if indeed any one such exists. Then we find that all other solutions $\bfy$ satisfy the
 system
\begin{equation}\label{3.9}
\sum_{i=1}^r \left((p^ay_i\zet^{-1}+1)^j-(p^aw_i\zet^{-1}+1)^j\right)\equiv 0
\pmod{p^{tb+\ome a}}\quad (\rho\le j\le k).
\end{equation}

\par Next we make use of Lemma \ref{lemma3.2} just as in the corresponding argument of the proof of
 \cite[Lemmata 3.3 to 3.6]{Woo2013}. Consider an index $j$ with $\rho\le j\le k$, and apply the latter
 lemma with $\alp=\rho-1$ and $\bet=j-\rho+1$. We find that there exist integers $c_{j,l}$ 
$(\rho-1\le l\le j)$ and $d_{j,m}$ $(j-\rho+1\le m\le j)$, depending at most on $j$ and $k$, and with
 $d_{j,j-\rho+1}\ne 0$, for which one has the polynomial identity
\begin{equation}\label{3.10}
c_{j,\rho-1}+\sum_{l=\rho}^jc_{j,l}(x+1)^l=\sum_{m=j-\rho+1}^jd_{j,m}x^m.
\end{equation}
Since we may assume $p$ to be large, moreover, we may suppose that $p\nmid d_{j,j-\rho+1}$. Thus,
 by multiplying the equation (\ref{3.10}) through by the multiplicative inverse of $d_{j,j-\rho+1}$ modulo
 $p^{tb+\ome a}$, we see that there is no loss in supposing that 
$d_{j,j-\rho+1}\equiv 1\mmod{p^{tb+\ome a}}$. Taking suitable linear combinations of the
 congruences comprising (\ref{3.9}), therefore, we deduce that any solution of this system satisfies
$$(\zet^{-1}p^a)^{j-\rho+1}\sum_{i=1}^r (\psi_j(y_i)-\psi_j(w_i))\equiv 0
\pmod{p^{tb+\ome a}}\quad (\rho\le j\le k),$$
in which
$$\psi_j(z)=z^{j-\rho+1}+\sum_{m=j-\rho+2}^jd_{j,m}(\zet^{-1}p^a)^{m-j+\rho-1}z^m.$$
We note for future reference that when $a\ge 1$, one has
\begin{equation}\label{3.11}
\psi_j(z)\equiv z^{j-\rho+1}\mmod{p}.
\end{equation}

\par Denote by $\calD_2(\bfu)$ the set of $\calR(tb-a)$-equivalence classes of solutions of the system
 of congruences
$$\sum_{i=1}^r \psi_j(y_i)\equiv u_j\mmod{p^{tb+\ome a-(j-\rho+1)a}}\quad (\rho\le j\le k),$$
with $1\le \bfy\le p^{tb+(\ome-1)a}$ satisfying \eqref{3.7}. Then we have shown thus far that
\begin{equation}\label{3.12}
\text{card}(\calD_1(\bfn))\le \max_{1\le \bfu\le p^{tb+\ome a}}\text{card}(\calD_2(\bfu)).
\end{equation}
Let $\calD_3(\bfv)$ denote the set of solutions of the system
\begin{equation}\label{3.13}
\sum_{i=1}^r \psi_j(y_i)\equiv v_j\mmod{p^{tb-a}}\quad (\rho\le j\le k),
\end{equation}
with $1\le \bfy\le p^{tb-a}$ satisfying \eqref{3.7}. 
For $\rho \le j\le k$, let 
\[
\tau(j)=\min \{ tb-a,tb+\omega a-(j-\rho+1)a\}.
\]
From \eqref{1.6} we see that $\tau(k)=tb+\omega a - ra \le tb - a$, and we obtain
\begin{align}
\text{card}(\calD_2(\bfu))&\le \sum_{\substack{1\le v_\rho\le p^{tb-a}\\ v_\rho\equiv
 u_\rho\mmod{p^{\tau(\rho)}}}}\cdots \sum_{\substack{1\le v_k\le p^{tb-a}\\ v_k\equiv
 u_k\mmod{p^{\tau(k)}}}}\text{card}(\calD_3(\bfv))\notag \\
&\le (p^a)^{\frac{1}{2}(r-\ome)(r-\ome-1)}\max_{1\le \bfv\le p^{tb-a}}
\text{card}(\calD_3(\bfv)).\label{3.14}
\end{align}

\par By combining (\ref{3.6}), (\ref{3.12}) and (\ref{3.14}), we discern at this point that
\begin{align}
\text{card}(\calC_{a,b}^{r,t}(\bfm;\xi,\eta))&\le (p^b)^\mu(p^a)^{(t-\rho+1)\ome +\frac{1}{2}
(r-\ome)(r-\ome-1)}\max_{1\le \bfv\le p^{tb-a}}\text{card}(\calD_3(\bfv))\notag \\
&=p^{\mu b+\nu a}\max_{1\le \bfv\le p^{tb-a}}\text{card}(\calD_3(\bfv)).\label{3.15}
\end{align}
It remains now only to bound the number of solutions of the system of congruences (\ref{3.13}) lying in
 the set $\calD_3(\bfv)$. Define the determinant
$$J(\bfpsi;\bfx)=\mathrm{det}\left( \psi'_{\rho+l-1}(x_i)\right)_{1\le i,l\le r}.$$
In view of \eqref{3.11}, one has $\psi'_{\rho+l-1}(y_i)\equiv  ly_i^{l-1}\mmod{p}$. It follows from
 \eqref{3.7} that
$$\text{det}(y_i^{l-1})_{1\le i,l\le r}=\prod_{1\le i<m\le r}(y_i-y_m)\not\equiv 0\mmod{p},$$
so that, since $p>k$, we have $(J(\bfpsi;\bfy),p)=1$.  We therefore deduce from Lemma \ref{lemma3.1}
 that
$$\text{card}(\calD_3(\bfv))\le \rho(\rho+1)\cdots k\le k!,$$
and thus the conclusion of the lemma when $a\ge 1$ follows at once from (\ref{3.2}) and (\ref{3.15}).

\par The proof presented above requires only small modifications when $a=0$. In this case, we denote by
 $\calD_1(\bfn;\eta)$ the set of $\calR(tb)$-equivalence classes of solutions of the system of congruences 
\eqref{3.5} with $1\le \bfz\le p^{kb}$ and $\bfz\equiv \bfxi\mmod{p}$ for some $\bfxi\in \Xi_0^r(0)$, 
and for which in addition $z_i\not\equiv \eta\mmod{p}$ for $1\le i\le r$. Then as in the opening 
paragraph of our proof, it follows from \eqref{3.1} that
\begin{equation}\label{3.16}
\mathrm{card}(\calC_{0,b}^{r,t}(\bfm;0,\eta))\le p^{\mu b}\max_{1\le \bfn\le p^{tb}}\mathrm{card}
(\calD_1(\bfn;\eta)).
\end{equation}
But $\text{card}(\calD_1(\bfn;\eta))=\text{card}(\calD_1(\bfn;0))$, and $\text{card}(\calD_1(\bfn;0))$ 
counts the solutions of the system of congruences
$$\sum_{i=1}^r y_i^j\equiv n_j\mmod{p^{tb}}\quad (\rho\le j\le k),$$
with $1\le \bfy\le p^{tb}$ satisfying (\ref{3.7}), and in addition $p\nmid y_i$ $(1\le i\le r)$. Write
$$J(\bfy)=\mathrm{det}\left( (\rho+j-1) y_i^{\rho+j-2}\right)_{1\le i,j\le r}.$$
Then, since $p>k$, we have 
$$J(\bfy) = \frac{k!}{(\rho-1)!}(y_1\cdots y_r)^{\rho-1}\prod_{1\le i<j\le r} (y_i-y_j) 
\not\equiv 0\mmod{p}.$$
We therefore conclude from Lemma \ref{lemma3.1} that
$$\text{card}(\calD_1(\bfn;0))\le \rho(\rho+1)\cdots k\le k!.$$
In view of (\ref{3.3}), the conclusion of the lemma therefore follows from (\ref{3.16}) when $a=0$.
\end{proof}

%
\section{The conditioning process}
%

We follow the previous treatments of \cite{Woo2012a} and \cite{Woo2013} in seeking next to bound the
 mean value $I_{a,b}(X;\xi,\eta)$ in terms of analogous mean values $K_{a,b+h}(X;\xi,\zeta)$, in which
 variables are arranged in ``non-singular'' blocks. We deviate from these earlier treatments, however, by
sacrificing some of the strength of these prior results in order to simplify the proofs. In particular, we are
 able in this way to avoid introducing coefficient $r$-tuples from $\{1,-1\}^r$ within the conditioned
 blocks of variables.

\begin{lemma}\label{lemma4.1} Let $a$ and $b$ be integers with $b>a\ge 1$. Then one has
$$I_{a,b}(X)\ll K_{a,b}(X)+M^{2s/3}I_{a,b+1}(X).$$
\end{lemma}

\begin{proof} Fix integers $\xi$ and $\eta$ with $\eta\not\equiv \xi\mmod{p}$. Let $T_1$ denote the
 number of solutions $\bfx$, $\bfy$, $\bfv$, $\bfw$ of the system \eqref{2.9} counted by 
$I_{a,b}(X;\xi,\eta)$ in which $v_1,\ldots,v_s$ together occupy at least $r$ distinct residue classes
 modulo $p^{b+1}$, and let $T_2$ denote the corresponding number of solutions in which 
$v_1,\ldots,v_s$ together occupy at most $r-1$ distinct residue classes modulo $p^{b+1}$. Then
\begin{equation}\label{4.1}
I_{a,b}(X;\xi,\eta)=T_1+T_2.
\end{equation}

\par We first estimate $T_1$.  Recall the definitions (\ref{2.6}), (\ref{2.7}) and (\ref{2.8}). Then by
 orthogonality and \RH, one finds that
\begin{align}
T_1&\le \binom{s}{r} \oint \left| \grF_a(\bfalp;\xi) \right|^2 \grF_b(\bfalp;\eta) \grf_b(\bfalp;\eta)^{s-r}
 \grf_b(-\bfalp;\eta)^s \d \bfalp \notag \\
&\ll \(K_{a,b}(X;\xi,\eta)\)^{1/(2t)} \(I_{a,b}(X;\xi,\eta)\)^{1-1/(2t)}.\label{4.2}
\end{align}
Next, we estimate $T_2$.  In view of the assumptions \eqref{1.6}, one has $s=rt\ge 2r>2(r-1)$.
 Consequently, there is an integer $\zeta\equiv \eta\pmod{p^b}$ having the property that at least three of
 the variables $v_1,\ldots,v_s$ are congruent to $\zeta$ modulo $p^{b+1}$. Hence, again recalling the
 definitions (\ref{2.7}) and (\ref{2.8}), one finds by orthogonality in combination with \RH that
\begin{align}
T_2 &\le \binom{s}{3} \sum_{\substack{1\le \zeta\le p^{b+1}\\ \zeta\equiv\eta\pmod{p^b}}}\oint 
|\grF_a(\bfalp;\xi)|^2 \grf_{b+1}(\bfalp;\zeta)^3 \grf_b(\bfalp;\eta)^{s-3}\grf_b(-\bfalp;\eta)^s\d\bfalp
 \notag \\
&\ll M\(I_{a,b}(X;\xi,\eta)\)^{1-3/(2s)} \( I_{a,b+1}(X) \)^{3/(2s)}.\label{4.3}
\end{align}

\par By substituting (\ref{4.2}) and (\ref{4.3}) into (\ref{4.1}), and recalling (\ref{2.11}) and
 (\ref{2.12}), we therefore conclude that
\begin{align*}
I_{a,b}(X)\ll &\, (K_{a,b}(X))^{1/(2t)}(I_{a,b}(X))^{1-1/(2t)}\\
&\, +M(I_{a,b}(X))^{1-3/(2s)}(I_{a,b+1}(X))^{3/(2s)},
\end{align*}
whence
$$I_{a,b}(X)\ll K_{a,b}(X)+M^{2s/3}I_{a,b+1}(X).$$
This completes the proof of the lemma.
\end{proof}

Repeated application of Lemma \ref{lemma4.1}, together with a trivial bound for the mean value
 $K_{a,b+H}(X)$ when $H$ is large enough, yields a relation suitable for iterating the efficient
 congruencing process.

\begin{lemma}\label{lemma4.2} Let $a$ and $b$ be integers with $1\le a<b$, and put $H=15(b-a)$.
 Suppose that $b+H\le (2\tet)^{-1}$. Then there exists an integer $h$ with $0\le h<H$ having the
 property that
$$I_{a,b}(X)\ll (M^h)^{2s/3}K_{a,b+h}(X)+(M^H)^{-s/4} (X/M^b)^{2s}(X/M^a)^{\lam-2s}.$$
\end{lemma}

\begin{proof}
By repeated application of Lemma \ref{lemma4.1}, we derive the upper bound
\be\label{4.4}
I_{a,b}(X)\ll \sum_{h=0}^{H-1} (M^h)^{2s/3}K_{a,b+h}(X)+(M^H)^{2s/3}I_{a,b+H}(X).
\ee
On considering the underlying Diophantine systems, it follows from Corollary \ref{corollary2.2} that
\begin{align*}
I_{a,b+H}(X;\xi,\eta)&\le \oint |\grf_a(\bfalp;\xi)^{2r}\grf_{b+H}(\bfalp;\eta)^{2s}|\d\bfalp \\
 &\ll (J_{s+r}(X/M^a))^{r/(s+r)}(J_{s+r}(X/M^{b+H}))^{s/(s+r)}.
\end{align*}

\par Since $M^{b+H} = (X^\theta)^{b+H}\le X^{1/2}$, we deduce from
 \eqref{2.4} that
\begin{align*}
(M^H)^{2s/3}I_{a,b+H}(X)&\ll X^\del \left( (X/M^a)^{r/(s+r)} (X/M^{b+H})^{s/(s+r)}\right)^{\lam}
(M^H)^{2s/3}\\
&=X^\del (X/M^b)^{2s}(X/M^a)^{\lam-2s}M^\Ome,
\end{align*}
where
$$\Ome=\lam \(a - \frac{ar}{s+r} - \frac{bs}{s+r}\)+2s(b-a)+Hs\(\frac23-\frac{\lam}{s+r}\).$$
We recall from (\ref{2.1}) that $\lam\ge s+r$. Then the lower bound $b\ge a$ leads to the estimate
$$\Ome\le -s(b-a)\frac{\lam}{s+r}+2s(b-a)-\tfrac{1}{3}Hs\le s(b-a)-\tfrac{1}{3}Hs.$$
But $H=15(b-a)$, and so from (\ref{2.2}) we discern that
$$\Ome\le -\tfrac{4}{15}Hs\le -\del\tet^{-1}-\tfrac{1}{4}Hs.$$
We therefore arrive at the estimate
$$(M^H)^{2s/3}I_{a,b+H}(X)\ll (M^H)^{-s/4}(X/M^b)^{2s}(X/M^a)^{\lam-2s},$$
and the conclusion of the lemma follows on substituting this bound into \eqref{4.4}.
\end{proof}

\section{The efficient congruencing step} 

We next seek to convert latent congruence information within the mean value $K_{a,b}(X)$ into a form
 useful in subsequent iterations, this being achieved by using the work of \S3. We recall now the definitions 
of the coefficients $\mu$ and $\nu$ from (\ref{3.4}). The following generalises Lemmata 5.1, 5.2, 6.2 and
 6.3 of \cite{Woo2013}.

\begin{lemma}\label{lemma5.1} Suppose that $a$ and $b$ are integers with $0\le a<b\le \tet^{-1}$ and 
$b\ge (k-t-1)a$. Then one has
$$K_{a,b}(X)\ll M^{\mu b+\nu a}(M^{tb-a})^r \(J_{s+r}(X/M^b)\)^{1-1/t}\(I_{b,tb}(X)\)^{1/t}.$$
\end{lemma}

\begin{proof} Suppose first that $a\ge 1$. Consider fixed integers $\xi$ and $\eta$ with
$$1\le \xi\le p^a,\quad 1\le \eta\le p^b\quad \text{and}\quad \eta\not\equiv \xi\mmod{p}.$$
The quantity $K_{a,b}(X;\xi,\eta)$ counts
 integral solutions of the system \eqref{2.10} subject to the attendant conditions on $\bfx$, $\bfy$,
 $\bfv$, $\bfw$. As in the argument of the proof of \cite[Lemma 6.1]{Woo2012a}, an application of the
 Binomial Theorem shows that these solutions satisfy the system of congruences
\be\label{5.1}
\sum_{i=1}^r (x_i-\eta)^j\equiv \sum_{i=1}^r (y_i-\eta)^j\mmod{p^{jb}}\quad (1\le j\le k).
\ee
In the notation of \S3, it follows that for some $k$-tuple of integers $\bfm$, we have
 $[\bfx\mmod{p^{kb}}] \in \calB_{a,b}^{r}(\bfm;\xi,\eta)$ and $[\bfy\mmod{p^{kb}}] \in
 \calB_{a,b}^{r}(\bfm;\xi,\eta)$.   Writing
$$\grG_{a,b}(\bfalp;\xi,\eta;\bfm)=\sum_{\bftet \in \calB_{a,b}^{r}(\bfm;\xi,\eta)}
\prod_{i=1}^r\grf_{kb}(\bfalp;\tet_i),$$
we see from (\ref{2.10}) and (\ref{5.1}) that
$$K_{a,b}(X;\xi,\eta)=\sum_{m_{1}=1}^{p^{b}}\ldots \sum_{m_k=1}^{p^{kb}}\oint |\grG_{a,b} 
(\bfalp;\xi,\eta;\bfm)^2\grF_b(\bfalp;\eta)^{2t}|\d\bfalp .$$

\par We now partition the vectors in each set $\calB_{a,b}^{r}(\bfm;\xi,\eta)$ into equivalence
classes modulo $p^{tb}$ as in Section 3.
An application of Cauchy's inequality leads via Lemma \ref{lemma3.3} to the bound
\begin{align*}
|\grG_{a,b}(\bfalp;\xi,\eta;\bfm)|^2 &= \Bigg| \sum_{C\in \calC_{a,b}^{r,t}(\bfm;\xi,\eta)} \;
\sum_{\bftet\in C}\prod_{i=1}^r \grf_{kb}(\bfalp;\tet_i) \Bigg|^2 \\
&\le \text{card}(\calC_{a,b}^{r,t}(\bfm;\xi,\eta)) \sum_{C\in \calC_{a,b}^{r,t}(\bfm;\xi,\eta)} 
\Bigg| \sum_{\bftet\in C}\prod_{i=1}^r \grf_{kb}(\bfalp;\tet_i) \Bigg|^2\\
&\ll M^{\mu b+\nu a} \sum_{C\in \calC_{a,b}^{r,t}(\bfm;\xi,\eta)} 
\Bigg| \sum_{\bftet\in C}\prod_{i=1}^r \grf_{kb}(\bfalp;\tet_i) \Bigg|^2.
\end{align*}
Hence
$$K_{a,b}(X;\xi,\eta)\ll M^{\mu b+\nu a} \sum_\bfm \sum_{C\in \calC_{a,b}^{r,t}(\bfm;\xi,\eta)} 
\oint \Bigg| \sum_{\bftet\in C}\prod_{i=1}^r \grf_{kb}(\bfalp;\tet_i) \Bigg|^2
|\grF_b(\bfalp;\eta)|^{2t}\d\bfalp.$$
For each $k$-tuple $\bfm$ and equivalence class $C$, the integral above counts solutions 
of \eqref{2.10} with the additional constraints that $[\bfx \mmod{p^{kb}}] \in C$ and 
$[\bfy \mmod{p^{kb}}] \in C$.  In particular, $\bfx \equiv \bfy \pmod{p^{tb}}$.
Moreover, as the sets $\calB_{a,b}^{r}(\bfm;\xi,\eta)$ are disjoint for distinct vectors $\bfm$
(with $1\le m_j\le p^{jb}$ for each $j$), to each pair $(\bfx,\bfy)$ there corresponds at most 
one pair $(\bfm, C)$.  Hence, 
\[
 K_{a,b}(X;\xi,\eta)\ll M^{\mu b+\nu a} H,
\]
where $H$ is the number of solutions of \eqref{2.10} with the additional hypothesis that $\bfx\equiv \bfy
\pmod{p^{tb}}$.  It follows that
\[
 K_{a,b}(X;\xi,\eta)\ll M^{\mu b+\nu a} \ssum{1\le \bfzet \le p^{tb} \\ \bfzet \equiv \xi \pmod{p^a}} 
 \oint \Big( \prod_{i=1}^r |\grf_{tb}(\bfalp;\zet_i)|^2 \Big) |\grF_b(\bfalp;\eta)|^{2t}\d\bfalp.
\]

\par An application of \RH reveals that
\begin{align*}
\sum_{\substack{1\le \bfzet \le p^{tb}\\ \bfzet\equiv \xi\pmod{p^a}}} \prod_{i=1}^r 
|\grf_{tb}(\bfalp;\zet_i)|^2&=\Big( \sum_{\substack{1\le \zeta \le p^{tb}\\ 
\zeta\equiv \xi\pmod{p^a}} } |\grf_{tb}(\bfalp;\zet)|^2\Big)^r\\
&\le (p^{tb-a})^{r-1}\sum_{\substack{1\le \zet \le p^{tb}\\ \zet\equiv \xi\pmod{p^a}}}
 |\grf_{tb}(\bfalp;\zet)|^{2r},
\end{align*}
and so it follows that
\be\label{5.2}
K_{a,b}(X;\xi,\eta)\ll M^{\mu b+\nu a}(M^{tb-a})^r\max_{\substack{1\le \zet \le p^{tb}\\ 
\zet\equiv \xi\mmod{p^a}}}\oint |\grf_{tb}(\bfalp;\zet)^{2r}\grF_b(\bfalp;\eta)^{2t}|\d\bfalp.
\ee
Next we apply \RH to the integral on the right hand side of (\ref{5.2}) to obtain
$$\oint |\grf_{tb}(\bfalp;\zet)^{2r}\grF_b (\bfalp;\eta)^{2t}|\d\bfalp \le U^{1-1/t}\(I_{b,tb}(X;\eta,
\zet)\)^{1/t},$$
where, on considering the underlying Diophantine system and using Lemma \ref{lemma2.1}, one has
$$U=\oint |\grF_b(\bfalp;\eta)|^{2t+2}\d\bfalp \le \oint |\grf_b(\bfalp;\eta)|^{2s+2r}\d\bfalp
\ll J_{s+r}(X/M^b).$$
Notice that since $\eta\not\equiv \xi\mmod{p}$ and $\zet\equiv \xi\mmod{p^a}$ with $a\ge 1$, one has
 $\zet\not\equiv \eta\mmod{p}$. Then we have $I_{b,tb}(X;\eta,\zet)\le I_{b,tb}(X)$, and so when 
$a\ge 1$ the conclusion of the lemma follows from (\ref{5.2}).

When $a=0$, we must modify the argument slightly. In this case, from \eqref{2.15} and \eqref{2.16} we
 find that
$$K_{0,b}(X)=\max_{1\le \eta\le p^b}\oint |\grF(\bfalp;\eta)^2\grF_b (\bfalp;\eta)^{2t}|\d\bfalp .$$
The desired conclusion then follows by pursuing the proof given above in the case $a\ge 1$, noting that
 the definition of $\grF(\bfalp;\eta)$ ensures that the variables resulting from the congruencing argument
 will avoid the congruence class $\eta$ modulo $p$. This completes the proof of the lemma.
\end{proof}

By applying Lemmata \ref{lemma4.2} and \ref{lemma5.1} in tandem, we obtain a sequence of inequalities
 for the quantities $K_{c,d}(X)$. Recall the definition of $\zero$ from (\ref{2.21}).

\begin{lemma}\label{lemma5.2}
Suppose that $a$ and $b$ are integers with $0\le a<b\le (32t\tet)^{-1}$ and $b\ge ta$. In addition, put
 $H=15(t-1)b$ and $g=b-ta$. Then there exists an integer $h$, with $0\le h<H$, having the property that
\begin{align*}
\llbracket K_{a,b}(X)\rrbracket \ll &\,X^\del \Big( (M^{g-4h/3})^s\llbracket K_{b,tb+h}(X)\rrbracket
 \Big)^{1/t}(X/M^b)^{\zero (1-1/t)}\\
&\,+(M^H)^{-r/6}(X/M^b)^\zero.
\end{align*}
\end{lemma}

\begin{proof} Recall the notational conventions \eqref{2.18} and \eqref{2.19}. The hypotheses $b\ge ta$
 and \eqref{1.6} imply that $b\ge (k-t-1)a$. Then it follows from Lemma \ref{lemma5.1} in combination 
with \eqref{2.4} that
\begin{equation}\label{5.3}
\llbracket K_{a,b}(X)\rrbracket \ll X^\del M^\ome \llbracket I_{b,tb}(X)\rrbracket^{1/t}
(X/M^b)^{\zero (1-1/t)},
\end{equation}
in which we have written
$$\ome=\mu b+\nu a+r(tb-a)+(2r-\kap)(a-b)-2s(t-1)b/t.$$
On recalling that $s=rt$ and noting the definition \eqref{1.7} of $\kap$, one finds that
\begin{align*}
\ome=\kap(b-a)-&(rt-\tfrac{1}{2}(t+r-k)(t+r-k-1))b\\
&+(r+\tfrac{1}{2}(t+r-k)(k+r-t-1))a,
\end{align*}
whence
$$\ome=\left( r-\frac{(t+r-k)(r-1)}{t-1}\right) (b-ta)\le rg.$$
The hypothesized upper bound on $b$ implies that $tb+H \le 16tb\le (2\tet)^{-1}$. We may therefore
 apply Lemma \ref{lemma4.2} to show that for some integer $h$ with $0\le h < H$, one has
$$\llbracket I_{b,tb}(X)\rrbracket \ll (M^h)^{-4s/3}\llbracket K_{b,tb+h}(X)\rrbracket 
+(M^H)^{-s/4}(X/M^b)^\zero .$$
We therefore deduce from (\ref{5.3}) that
\begin{align}
\llbracket K_{a,b}(X)\rrbracket \ll &\, X^\del (X/M^b)^{\zero(1-1/t)}M^{rg-4rh/3}\llbracket
 K_{b,tb+h}(X)\rrbracket^{1/t}\notag \\
&\, +X^\del M^{rg-rH/4}(X/M^b)^\zero .\label{5.4}
\end{align}
But in view of the hypotheses \eqref{1.6}, one has $t\ge 2$ and hence
$$H=15(t-1)b\ge 15b\ge 15g.$$
Then on recalling (\ref{2.2}), we find that
$$X^\del (M^r)^{g-H/4}\le M^{\del \tet^{-1}}(M^{rH})^{1/15-1/4}\le M^{-rH/6}.$$
The conclusion of the lemma therefore follows from (\ref{5.4}).
\end{proof}

The following crude upper bound for $K_{a,b}(X)$ is a useful addition to our arsenal when $b$ is very
 large.

\begin{lemma}\label{lemma5.3}
Suppose that $a$ and $b$ are integers with $0\le a<b\le (2\tet)^{-1}$. Then provided that $\zero\ge 0$,
 one has
$$\llbracket K_{a,b}(X)\rrbracket \ll X^{\zero+\del}(M^{b-a})^s.$$
\end{lemma}

\begin{proof} 
On considering the underlying Diophantine equations, we deduce from Corollary \ref{corollary2.2} that
$$K_{a,b}(X)\ll (J_{s+r}(X/M^a))^{r/(s+r)}(J_{s+r}(X/M^b))^{s/(s+r)},$$
so that \eqref{2.4}, (\ref{2.17}), (\ref{2.19}) and \eqref{2.21} yield the relation
\begin{align*}
\llbracket K_{a,b}(X)\rrbracket&\ll \frac{X^\del \left( (X/M^a)^{r/(s+r)}
(X/M^b)^{s/(s+r)}\right)^{2s+2r-\kap+\zero}}{(X/M^b)^{2s}(X/M^a)^{2r-\kap}}\\
&\le X^{\zero+\del}(M^{b-a})^{\kap s/(s+r)}.
\end{align*}
In view of (\ref{1.7}), one has $\kap \le s+r$, and thus the proof of the lemma is complete.
\end{proof}

%
\section{The pre-congruencing step}
%

In order to ensure that the variables in the auxiliary mean values that we consider are appropriately
 configured, we must expend some additional effort initiating the iteration in a pre-congruencing step. It is
 at this point that we fix the prime $p$ once and for all. Although we follow the argument of 
\cite[Lemma 6.1]{Woo2013} in broad strokes, we are able to obtain some simplification by weakening our
 conclusions inconsequentially.

\begin{lemma}\label{lemma6.1}
There exists a prime number $p$ with $M<p\le 2M$, and an integer $h$ with $h\in\{0,1,2,3\}$, for which
 one has
$$J_{s+r}(X)\ll M^{2s+2sh/3}K_{0,1+h}(X).$$
\end{lemma}

\begin{proof}
The mean value $J_{s+r}(X)$ counts the number of integral solutions of the system
\be\label{6.1}
 \sum_{i=1}^{s+r} (x_i^j-y_i^j)=0 \quad (1\le j\le k),
\ee
with $1\le \bfx,\bfy\le X$. Let $T_1$ denote the number of these solutions with either two of 
$x_1,\ldots,x_{s+r}$ equal or two of $y_1,\ldots,y_{s+r}$ equal, and let $T_2$ denote the
 corresponding number of solutions with $x_1,\ldots,x_{s+r}$ distinct and $y_1,\ldots,y_{s+r}$ distinct.
 Then we have $J_{s+r}(X)=T_1+T_2$.\par

Suppose first that $T_1\ge T_2$. Then by considering the underlying Diophantine systems, it follows from
 \RH that
\begin{align*}
J_{s+r}(X) &\le 2T_1\le 4\binom{s+r}{2}\oint \left| \grf_0(\bfalp;0)^{2s+2r-2} \grf_0(2\bfalp;0)
 \right|\d \bfalp \\
&\ll \(\oint |\grf_0(\bfalp;0)|^{2s+2r}\d \bfalp \)^{1-1/(s+r)} \(\oint |\grf_0(2\bfalp;0)|^{2s+2r}
\d \bfalp \)^{1/(2s+2r)}\\
&= \(J_{s+r}(X) \)^{1-1/(2s+2r)}.
\end{align*}
Consequently, one has $J_{s+r}(X)\ll 1$, which contradicts the lower bound \eqref{2.3} if $X=X_\ell$ is
 large enough. We may therefore suppose that $T_1<T_2$, and hence that $J_{s+r}(X)\le 2T_2$.\par

Given a solution $\bfx,\bfy$ of \eqref{6.1} counted by $T_2$, let 
$$D(\bfx,\bfy)=\prod_{1\le i < j\le s+r} (x_i-x_j)(y_i-y_j).$$
Also, let $\calP$ denote a set of  $\lceil (s+r)^2\theta^{-1}\rceil$ prime numbers in $(M,2M]$. That such
 a set of primes exists for large enough $X$ is a consequence of the Prime Number Theorem. From the
 definition of $T_2$, we have $D(\bfx,\bfy) \ne 0$ and
$$|D(\bfx,\bfy)|<X^{(s+r)^2}\le M^{\text{card}(\calP)}.$$
We therefore find that for some $p\in \calP$ one must have $p\nmid D(\bfx,\bfy)$. Denote by $T_2(p)$
 the number of solutions of (\ref{6.1}) counted by $J_{s+r}(X)$ in which $x_1,\ldots,x_{s+r}$ are
 distinct modulo $p$ and likewise $y_1,\ldots,y_{s+r}$ are distinct modulo $p$. Then we have shown
 thus far that
$$J_{s+r}(X)\le 2T_2\le 2\sum_{p\in \calP}T_2(p),$$
whence for some prime number $p\in \calP$, one has
\begin{equation}\label{6.2}
J_{s+r}(X)\le 2\lceil (s+r)^2\tet^{-1}\rceil T_2(p).
\end{equation}

We next introduce some notation with which to consider more explicitly the residue classes modulo $p$ of
 a given solution $\bfx$, $\bfy$ counted by $T_2(p)$. Let $\bfeta$ and $\bfzet$ be $s$-tuples with 
$1\le \bfeta,\bfzet\le p$  satisfying the condition that for $1\le i\le s$, one has $x_i\equiv \eta_i\pmod{p}$
 and $y_i\equiv \zeta_i\pmod{p}$. Recall the notation introduced prior to the definition (\ref{2.13}). Then
 since $x_1,\ldots ,x_{s+r}$ are distinct modulo $p$, it follows that $(x_{s+1},\ldots ,x_{s+r})\in
 \Xi(\bfeta)$, and likewise one finds that $(y_{s+1},\ldots ,y_{s+r})\in \Xi(\bfzet)$. Then on considering
 the underlying Diophantine systems, we obtain the relation
$$T_2(p)\le \sum_{1\le \bfeta,\bfzet\le p}\oint \Bigl( \prod_{i=1}^s\grf_1(\bfalp;\eta_i)
\grf_1(-\bfalp;\zet_i)\Bigr) \grF(\bfalp;\bfeta) \grF(-\bfalp;\bfzet)\d\bfalp .$$
Write
$$\grI(\bftet,\psi)=\oint \left|\grF(\bfalp;\bftet)^2\grf_1(\bfalp;\psi)^{2s}\right|\d\bfalp .$$
Then by applying \RH, and again considering the underlying Diophantine systems, we discern that
\begin{align*}
T_2(p)&\le \sum_{1\le \bfeta,\bfzet\le p}\prod_{i=1}^s \(\grI(\bfeta,\eta_i)
\grI(\bfzet,\zeta_i)\)^{1/(2s)}\\
&\le \sum_{1\le \bfeta,\bfzet\le p}\prod_{i=1}^s \( \grI(\eta_i,\eta_i)\grI(\zet_i,\zeta_i)\)^{1/(2s)}.
\end{align*}
Hence, on recalling the definition (\ref{2.14}), we obtain the upper bound
\begin{align}
T_2(p)&\le p^{2s} \max_{1\le \eta\le p}  \oint \left|\grF(\bfalp;\eta)^2 \grf_1(\bfalp;\eta)^{2s}\right|
\d\bfalp \notag \\
&= p^{2s} \max_{1\le \eta\le p} \Itil_1(X;\eta).\label{6.3}
\end{align}

The mean value $\Itil_c(X;\eta)$ counts the number of integral solutions of the system \eqref{2.9} with
$$1\le \bfx,\bfy,\bfv,\bfw\le X,\quad \bfv\equiv \bfw\equiv \eta\mmod{p^c},$$
and with
$$[\bfx \mmod{p}]\in \Xi(\eta)\quad \text{and}\quad [\bfy \mmod{p}]\in \Xi(\eta).$$
Let $T_3$ denote the number of such solutions in which the $s$ integers $v_1,\ldots ,v_s$ together
 occupy at least $r$ distinct residue classes modulo $p^{c+1}$, and let $T_4$ denote the corresponding
 number of solutions in which these integers together lie in at most $r-1$ distinct residue classes modulo
 $p^{c+1}$. Then $\Itil_c(X;\eta)=T_3+T_4$. By an argument similar to that leading to \eqref{4.2}, we
 obtain the bound
\begin{align}
T_3&\ll \oint \left| \grF(\bfalp;\eta) \right|^2 \grF_c(\bfalp;\eta) \grf_c(\bfalp;\eta)^{s-r}
\grf_c(-\bfalp;\eta)^s\d\bfalp \notag \\
&\le\(\oint \left| \grF(\bfalp;\eta)^2 \grF_c(\bfalp;\eta)^{2t}\right| \d\bfalp\)^{1/(2t)}\( \oint 
\left| \grF(\bfalp;\eta)^2 \grf_c(\bfalp;\eta)^{2s} \right|\d\bfalp \)^{1-1/(2t)}\notag \\
&\le \(\Ktil_c(X;\eta)\)^{1/(2t)} \(\Itil_c(X;\eta)\)^{1-1/(2t)}.\label{6.4}
\end{align}
Also, since $s\ge 2r > 2(r-1)$, the argument leading to \eqref{4.3} implies that
\be\label{6.5}
T_4 \ll M  \(\Itil_c(X;\eta)\)^{1-3/(2s)} \( \max_{1\le \zeta\le p^{c+1}}\Itil_{c+1}(X;\zeta)\)^{3/(2s)}.
\ee
Then by combining \eqref{6.4} and \eqref{6.5} to bound $\Itil_c(X;\eta)$, we infer that
\be\label{6.6}
 \Itil_c(X;\eta) \ll \Ktil_c(X;\eta) + M^{2s/3} \max_{1\le \zeta\le p^{c+1}}\Itil_{c+1}(X;\zeta).
\ee

\par We now iterate \eqref{6.6} to bound $\Itil_1(X;\eta)$, thereby deducing from (\ref{6.2}),
 (\ref{6.3}) and the definition \eqref{2.16} that
\begin{align}
J_{s+r}(X)&\ll T_2(p)\notag \\
&\ll \max_{0\le h\le 3} M^{2s}(M^h)^{2s/3}K_{0,1+h}(X)+M^{2s+8s/3} 
\max_{1\le \zeta\le p^5}\Itil_5(X;\zeta).
\label{6.7}
\end{align}
By considering the underlying Diophantine systems, we deduce from \eqref{2.13} and \eqref{2.14} via
 Corollary \ref{corollary2.2} that
\begin{align*}
\Itil_5(X;\zeta)&\le \oint |\grf_0(\bfalp;0)^{2r}\grf_5(\bfalp;\zet)^{2s}|\d\bfalp \\
&\ll \(J_{s+r}(X)\)^{r/(s+r)} \(J_{s+r}(X/M^5)\)^{s/(s+r)}.
\end{align*}
Now \eqref{6.7} implies either that
\begin{equation}\label{6.8}
J_{s+r}(X)\ll M^{2s+2sh/3}K_{0,1+h}(X)
\end{equation}
for some index $h\in\{0,1,2,3\}$, so that the conclusion of the lemma holds, or else that 
$$J_{s+r}(X)\ll M^{14s/3}(J_{s+r}(X))^{r/(s+r)} (J_{s+r}(X/M^5))^{s/(s+r)}.$$
In the latter case, since $\lam\ge s+r$, we obtain the upper bound
\begin{align*}
J_{s+r}(X)&\ll M^{14(s+r)/3} J_{s+r}(X/M^5) \ll M^{14(s+r)/3} (X/M^5)^{\lam+\del}\\
&\ll X^{\lam+\del} M^{-(s+r)/3}.
\end{align*}
Invoking the definition \eqref{2.2} of $\del$, we find that $J_{s+r}(X)\ll X^{\lam-2\del}$, contradicting
 the lower bound \eqref{2.3} if $X=X_\ell$ is large enough. We are therefore forced to accept the former
 upper bound (\ref{6.8}), and hence the proof of the lemma is complete.\end{proof}

%
\section{The iterative process} 
%

By first applying Lemma \ref{lemma6.1}, and following up with repeated application of Lemma
 \ref{lemma5.2}, we are able to bound $J_{s+r}(X)$ in terms of quantities of the shape $K_{c,d}(X)$, in
 which $c$ and $d$ pass through an increasing sequence of integral values. In this section we explore this
 iterative process, and ultimately establish Theorem \ref{theorem1.3}.

\begin{lemma}\label{lemma7.1} 
Suppose $\zero\ge 0$.
Let $a$ and $b$ be integers with $0\le a<b\le (32t\tet)^{-1}$ and $b\ge ta$, and put $g=b-ta$. Suppose
 that there are real numbers $\psi$, $c$ and $\gam$, with
$$0\le c\le (2\delta)^{-1}\tet,\quad \gamma\ge -rb\quad \text{and}\quad \psi\ge 0,$$
such that
\begin{equation}\label{7.1}
X^{\zero}M^{\zero\psi}\ll X^{c\del}M^{-\gam}\llbracket K_{a,b}(X)\rrbracket .
\end{equation}
Then, for some integer $h$ with $0\le h\le 15(t-1)b$, one has
$$X^\zero M^{\zero\psi'}\ll X^{c'\del}M^{-\gam'}\llbracket K_{b,tb+h}(X)\rrbracket ,$$
where
$$\psi'=t\psi+(t-1)b, \quad c'=t(c+1), \quad \gam'=t\gam+\tfrac{4}{3}sh-sg.$$
\end{lemma}

\begin{proof}
 From Lemma \ref{lemma5.2},
 there exists an integer $h$ with $0\le h<15(t-1)b$ with the property that
\begin{align*}
\llbracket K_{a,b}(X)\rrbracket \ll X^\del M^{rg}\( M^{-4sh/3}\llbracket K_{b,tb+h}
(X)\rrbracket\right)^{1/t}
(X/M^b)^{\zero (1-1/t)}&\\
+(M^{15(t-1)b})^{-r/6}&(X/M^b)^{\zero}.
\end{align*}
Consequently, from the hypothesised bound \eqref{7.1} we infer that
\begin{align*}
X^\zero M^{\zero\psi}\ll X^{(c+1)\del}M^{-\gam +rg}\( M^{-4sh/3}\llbracket 
K_{b,tb+h}(X)\rrbracket\right)^{1/t}&(X/M^b)^{\zero(1-1/t)}\\
&\,+X^{c\delta} M^{-\gamma-2rb}X^\zero.
\end{align*}
By hypothesis, we have $X^{c\del} \le M^{1/2}$,
whence $X^{c\delta} M^{-\gamma-2rb} \le M^{1/2-rb} \le M^{-1/2}$ and thus
$$X^{\zero/t}M^{\zero(\psi+(1-1/t)b)}\ll X^{(c+1)\del}M^{-\gam+rg-4rh/3}\llbracket 
K_{b,tb+h}(X)\rrbracket^{1/t}.$$
The conclusion of the lemma follows on raising left and right hand sides in the last inequality to the power
 $t$.
\end{proof}

\begin{lemma}\label{lemma7.2} We have $\zero\le 0$.
\end{lemma}

\begin{proof} Assume that $\zero>0$, for otherwise there is nothing to prove. We begin by noting that as
 a consequence of Lemma \ref{lemma6.1}, it follows from (\ref{2.17}) and
 (\ref{2.19}) that there exists an integer $h_{-1}\in\{0,1,2,3\}$ such that
$$\llbracket J_{s+r}(X)\rrbracket \ll (M^{h_{-1}})^{-4s/3}\llbracket K_{0,1+h_{-1}}(X)\rrbracket .$$
We therefore deduce from (\ref{2.20}) that
\begin{equation}\label{7.2}
X^\zero \ll X^\del \llbracket J_{s+r}(X)\rrbracket \ll X^\del (M^{h_{-1}})^{-4s/3}\llbracket 
K_{0,1+h_{-1}}(X)\rrbracket .
\end{equation}

\par Next we define sequences $(a_n)$, $(b_n)$, $(h_n)$, $(c_n)$, $(\gam_n)$, $(\psi_n)$ for 
$0\le n\le N$ in such a way that
\begin{equation}\label{7.3}
0\le h_{n-1}\le 15(t-1)b_{n-1}\quad (n\ge 1),
\end{equation}
and
\begin{equation}\label{7.4}
X^\zero M^{\zero\psi_n}\ll X^{c_n\del}M^{-\gam_n}\llbracket K_{a_n,b_n}(X)\rrbracket .
\end{equation}
Given a fixed choice for the sequence $(h_n)$, these sequences are defined by means of the relations
\begin{align}
a_{n+1}&=b_n\quad \text{and}\quad b_{n+1}=tb_n+h_n, \label{7.5}\\
\psi_{n+1}&=t\psi_n+(t-1)b_n,\label{7.6}\\
c_{n+1}&=t(c_n+1),\label{7.7}\\
\gam_{n+1}&=t\gam_n+\tfrac{4}{3}sh_n-s(b_n-ta_n).\label{7.8}
\end{align}
We put $a_0=0$, $b_0=1+h_{-1}$, $\psi_0=0$, $c_0=1$ and $\gam_0=\tfrac{4}{3}sh_{-1}$, so that
\eqref{7.4} holds with $n=0$ as a consequence of our initial choice of $h_{-1}$ together with (\ref{7.2}).
 We prove by induction that for each integer $n$ with $0\le n<N$, the sequence $(h_m)_{m=-1}^n$ may
 be chosen in such a way that
\be\label{7.9}
0\le a_n<b_n\le (32t\tet)^{-1}, \quad \psi_n\ge 0, \quad \gamma_n \ge -rb_n, \quad
0\le c_n\le (2\del)^{-1}\tet,
\ee
and so that (\ref{7.3}) and (\ref{7.4}) both hold with $n$ replaced by $n+1$.\par

Suppose that $0\le n<N$, and suppose also that (\ref{7.3}) and (\ref{7.4}) both hold for the index $n$. 
We have already shown such to be the case when $n=0$. We observe first that the relation (\ref{7.5})
plainly demonstrates that $b_n>a_n$ for all $n$. Moreover, from (\ref{7.3}) and (\ref{7.5}), we see that
 $b_{n+1}\le 16tb_n$ for all $n$. By induction, therefore, we deduce that $b_n\le 4(16t)^n$ whence, by
 invoking (\ref{2.2}) we find that $b_n\le (32t\tet)^{-1}$ for $0\le n<N$. It is also apparent from
 (\ref{7.6}) and (\ref{7.7}) that $c_n$ and $\psi_n$ are non-negative for all $n$. In addition, by iterating
 \eqref{7.7}, we have
\be\label{7.10}
c_n = t^n+t\left( \frac{t^n-1}{t-1}\right)\le 3t^n\quad (n\ge 0).
\ee
Thus, by reference to (\ref{2.2}) we see that $c_n\le (2\del)^{-1}\tet$ for $0\le n<N$.

In order to bound $\gam_n$, we begin by noting from (\ref{7.5}) that 
for $m\ge 1$,
$$h_m=b_{m+1}-tb_m\quad \text{and}\quad a_m=b_{m-1}.$$
 Then it follows from (\ref{7.8}) that for $m\ge 1$ one has
$$\gam_{m+1}-\tfrac{4}{3}sb_{m+1}+sb_m=t\( \gam_m - \tfrac{4}{3}sb_m + sb_{m-1} \).$$
By iterating this identity, we deduce that for $m\ge 1$, one has
$$\gam_m=\tfrac{4}{3}sb_m-sb_{m-1}+t^{m-1}\( \gam_1-\tfrac{4}{3}sb_1+sb_0 \).$$
On recalling that $b_0=1+h_{-1}$, $\gam_0=\frac{4}{3}sh_{-1}$ and $b_1=tb_0+h_0$, we discern
 first from (\ref{7.8}) that
\[
\gam_1=\tfrac{4}{3}st(b_0-1)+\tfrac{4}{3}s(b_1-tb_0)-s b_0=\tfrac{4}{3}s(b_1-t)-sb_0,
\]
and hence that
\be\label{7.11}
\gamma_m=\tfrac{4}{3}sb_m-sb_{m-1}-\tfrac43 st^m \quad (m\ge 1).
\ee
Finally, we find from \eqref{7.5} that $b_m\ge tb_{m-1} \ge t^{m}$ for $m\ge 1$, and hence
\[
 \gamma_m = \tfrac43 s (b_m - t^m) - sb_{m-1} \ge - sb_{m-1} \ge -rb_m.
\]
Collecting together this conclusion with those of the previous paragraph, we have shown that
\eqref{7.9} holds for $0\le n<N$.
\par

At this point in the argument, we may suppose that both (\ref{7.4}) and (\ref{7.9}) hold for the index
 $n$. An application of Lemma \ref{lemma7.1} therefore reveals that there exists an integer $h_n$
 satisfying the constraint implied by (\ref{7.3}) with $n$ replaced by $n+1$, for which the upper bound
 (\ref{7.4}) holds also with $n$ replaced by $n+1$. This completes the inductive step, so that in particular
 the upper bound (\ref{7.4}) holds for $0\le n\le N$.\par

We now exploit the bound just established.
Since we have $b_N\le 4(16t)^N\le (2\tet)^{-1}$, it is a
 consequence of Lemma \ref{lemma5.3} that
\begin{equation}\label{7.12}
\llbracket K_{a_N,b_N}(X)\rrbracket \ll X^{\zero+\del}(M^{b_N-b_{N-1}})^s .
\end{equation}
By combining \eqref{7.4} with \eqref{7.11} and \eqref{7.12}, we obtain the bound
\begin{align}
X^\zero M^{\zero \psi_N}&\ll X^{\zero+(c_N+1)\del}M^{(b_N-b_{N-1})s-\gam_N}\notag \\
&= X^{\zero+(c_N+1)\del}M^{(4s/3)t^N-(s/3)b_N}.\label{7.13}
\end{align}
By applying \eqref{7.10} and \eqref{2.2}, on the other hand, we have
$$X^{(c_N+1)\del}<M.$$
We therefore deduce from (\ref{7.13}) and the lower bound $b_N\ge t^N$ that
$$\zero \psi_N\le \tfrac{4}{3}st^N-\tfrac{1}{3}b_Ns+1 \le st^N+1.$$
In addition, a further application of the lower bound $b_n\ge t^n$ reveals that
$$\psi_{n+1}=t\psi_n+(t-1)b_n\ge t\psi_n+(t-1)t^n,$$
whence $\psi_N\ge N(t-1)t^{N-1}$. Thus we deduce that
$$\zero\le \frac{st^N+1}{N(t-1)t^{N-1}}\le \frac{3s}{N}.$$
Since $N$ may be taken arbitrarily large in terms of $s$, we are forced to conclude that $\zero\le 0$, and
 this completes the proof of the lemma.
\end{proof}

The conclusion of Theorem \ref{theorem1.3} is an immediate consequence of Lemma \ref{lemma7.2}. 
For the latter shows that when $s= rt$, then for each $\eps>0$ one has
$$J_{s+r}(X)\ll X^{2s+2r-\kap+\eps},$$
where $\kap$ is given by (\ref{1.7}).

%
\section{A mean value estimate for Weyl sums}
%

Our goal in this section is to establish a mean value estimate for one-dimensional Weyl sums that, in a 
sense, forms a hybrid between the treatments of \cite{For1995} and \cite[\S10]{Woo2012a}. This estimate 
permits the output from the efficient congruencing method to be more effectively transformed into a mean 
value estimate for one-dimensional Weyl sums.\par

Consider natural numbers $s$ and $m$ with $1\le m\le k$. When $q\in \dbN$ and $b\in \dbZ$, we define 
the quantity $I_{s,m}(X;q,b)$ to be the number of integral solutions of the system of equations
\begin{equation}\label{8.1}
\left.
\begin{aligned}
\sum_{i=1}^s\( (qx_i+b)^k-(qy_i+b)^k\)&=0,\\
\sum_{i=1}^s(x_i^j-y_i^j)&=0\quad (1\le j\le m-1),
\end{aligned}
\right\}
\end{equation}
with $0\le \bfx,\bfy\le X/q$. We begin by adapting the work of \cite[\S10]{Woo2012a} so as to estimate
 $I_{s,k-1}(X;q,b)$ on average over $q$. To assist with our discussion, we now define $\eta(s,k)$ to be
 the least positive number $\eta$ with the property that, whenever $X$ is sufficiently large in terms of $s$
 and
 $k$, one has
$$J_{s,k}(X)\ll_\eps X^{2s-\frac{1}{2}k(k+1)+\eta+\eps}.$$
Throughout this section and the following section, we adopt the convention that
whenever $\eps$
 appears in a statement, either implicitly or explicitly, we assert that the statement holds for each
 $\eps>0$. Note that the ``value'' of $\eps$ may consequently change from statement to statement. 
It is convenient to write
$$f(\bfalp;X)=\sum_{1\le x\le X}\er (\alp_1x+\ldots +\alp_kx^k).$$

\par We pause to recall a lemma on reciprocal sums.

\begin{lemma}\label{lemma8.1}
Suppose that $\del$ is a positive number, and that $\alp$ and $\bet$ are real numbers. Let $N$ and $R$
 be large real numbers, and write $B=N^{1+\del}+R^{1+\del}$. Then
$$\sum_{1\le z\le R}\min\{ N,\|z\alp +\bet\|^{-1}\}\ll B+(\log B)\sum_{1\le u\le BN^{-\del}}\min \{
 NR/u,\|u\alp\|^{-1}\}.$$
\end{lemma}

\begin{proof} This is \cite[Lemma 3.4]{Woo2000}.
\end{proof}

When $\calQ\subset \dbN$, write
$$\Tet_{s,k}(\calQ)=\sum_{q\in \calQ}\max_{(b,q)=1}I_{s,k-1}(X;q,b).$$

\begin{lemma}\label{lemma8.2}
Let $X$ denote a large positive number, and let $Q$ be a real number with $1<Q\le X^{(k-2)/(k-1)}$.
 Suppose that $\calQ\subseteq (2^{-k}Q,Q]$ is a set of natural numbers with 
$\text{card}(\calQ)\gg Q(\log Q)^{-k}$ satisfying the condition that for each $q\in \calQ$, one has
 $(q,k)=1$. Then for each natural number $s$, one has
$$\Tet_{s,k}(\calQ)\ll (X/Q)^{2s-\frac{1}{2}(k^2-k+2)+\eps}
\( (X/Q)^{\eta(s,k)-1}+(X/Q)^{\eta(s,k-1)}\).$$
\end{lemma}

\begin{proof} For the moment, consider fixed integers $q$ and $b$ with $(kb,q)=1$ and 
$2^{-k}Q<q\le Q$. Define $\Ups_k(X;h)=\Ups_k(X;h;q,b)$ to be the number of integral solutions of the
 Diophantine system
\begin{equation}\label{8.2}
\left.
\begin{aligned}
\sum_{i=1}^s\( (qx_i+b)^k-(qy_i+b)^k\)&=0,\\
\sum_{i=1}^s(x_i^{k-1}-y_i^{k-1})&=h,\\
\sum_{i=1}^s(x_i^j-y_i^j)&=0\quad (1\le j\le k-2),
\end{aligned}
\right\}
\end{equation}
with $0\le \bfx,\bfy\le X/q$. Then on considering the corresponding system (\ref{8.1}), we see that
\begin{equation}\label{8.3}
I_{s,k-1}(X;q,b)=\sum_{|h|\le s(X/q)^{k-1}}\Ups_k(X;h).
\end{equation}

\par Next, by applying an integer shift $z$ to the variables in the system (\ref{8.2}), we find that
 $\Ups_k(X;h)$ counts the number of integral solutions of the Diophantine system
$$\left.
\begin{aligned}
\sum_{i=1}^s\( (q(x_i-z)+b)^k-(q(y_i-z)+b)^k\)&=0,\\
\sum_{i=1}^s((x_i-z)^{k-1}-(y_i-z)^{k-1})&=h,\\
\sum_{i=1}^s((x_i-z)^j-(y_i-z)^j)&=0\quad (1\le j\le k-2),
\end{aligned}
\right\}$$
with $z\le \bfx,\bfy\le z+X/q$. By applying the Binomial Theorem, we find that $\bfx,\bfy$ satisfies this
 system of equations if and only if
\begin{equation}\label{8.4}
\left.
\begin{aligned}
\sum_{i=1}^s(x_i^j-y_i^j)&=0\quad (1\le j\le k-2),\\
\sum_{i=1}^s(x_i^{k-1}-y_i^{k-1})&=h,\\
q\sum_{i=1}^s\(x_i^k-y_i^k\)&=k(qz-b)h.
\end{aligned}
\right\}
\end{equation}
Notice that, in view of the hypothesis $(kb,q)=1$, the equation of degree $k$ in (\ref{8.4}) ensures that
 $q|h$. We write $g=h/q$, so that the condition $|h|\le s(X/q)^{k-1}$ in (\ref{8.3}) implies that 
$|g|\le sq^{-1}(X/q)^{k-1}$.\par

If we restrict the shifts $z$ to lie in the interval $1\le z\le X/q$, then we see that an upper bound for 
$\Ups_k(X;h)$ is given by the number of integral solutions of the system
$$\left.
\begin{aligned}
\sum_{i=1}^s(x_i^j-y_i^j)&=0\quad (1\le j\le k-2),\\
\sum_{i=1}^s(x_i^{k-1}-y_i^{k-1})&=qg,\\
\sum_{i=1}^s\(x_i^k-y_i^k\)&=k(qz-b)g,
\end{aligned}
\right\}$$
with $1\le \bfx,\bfy\le 2X/q$. On considering the underlying Diophantine system, we therefore deduce
 from (\ref{8.3}) that for each integer $z$ with $1\le z\le X/q$, the mean value $I_{s,k-1}(X;q,b)$ is
 bounded above by
$$\sum_{|g|\le sq^{-1}(X/q)^{k-1}}\oint |f(\bfalp;2^{k+1}X/Q)|^{2s}
\er(-k(qz-b)g\alp_k-qg\alp_{k-1})\d\bfalp .$$
Write
$$\psi_{q,b}(z;\alp_k,\alp_{k-1})=\min \{ q^{-1}(X/q)^{k-1},\|k(qz-b)\alp_k+q\alp_{k-1}\|^{-1}\}$$
and
\begin{equation}\label{8.5}
\Psi_{q,b}(\alp_k,\alp_{k-1})=\sum_{1\le z\le X/q}\psi_{q,b}(z;\alp_k,\alp_{k-1}).
\end{equation}
Then we obtain the estimate
\begin{align}
I_{s,k-1}(X;q,b)&\ll (X/q)^{-1}\sum_{1\le z\le X/q}\oint |f(\bfalp;2^{k+1}X/Q)|^{2s}\psi_{q,b}(z
;\alp_k,\alp_{k-1})\d\bfalp \notag \\
&=(X/q)^{-1}\oint |f(\bfalp;2^{k+1}X/Q)|^{2s}\Psi_{q,b}(\alp_k,\alp_{k-1})\d\bfalp .\label{8.6}
\end{align}

\par Our assumption that $1<Q\le X^{(k-2)/(k-1)}$ ensures that $X/q\le q^{-1}(X/q)^{k-1}$. Then by
 applying Lemma \ref{lemma8.1} with $\alp=kq\alp_k$, we deduce from (\ref{8.5}) that
\begin{align*}
\Psi_{q,b}(\alp_k,\alp_{k-1})\ll &\, q^{-1}(X/q)^{k-1+\eps}\\
&+X^\eps\sum_{1\le u\le 2q^{-1}(X/q)^{k-1}}\min\{ (qu)^{-1}(X/q)^k,\|kqu\alp_k\|^{-1}\}.
\end{align*}
Define
\begin{equation}\label{8.7}
\Phi(\alp_k,\alp_{k-1})=\sum_{q\in \calQ}\max_{(b,q)=1}\Psi_{q,b}(\alp_k,\alp_{k-1}).
\end{equation}
Then we arrive at the upper bound
\begin{align*}
\Phi(\alp_k,\alp_{k-1})\ll &\, X^{k-1+\eps}\sum_{2^{-k}Q<q\le Q}q^{-k}\\
&+X^\eps\sum_{1\le q\le Q}\sum_{1\le u\le 2q^{-1}(X/q)^{k-1}}\min\{ (qu)^{-1}(X/Q)^k,
\|kqu\alp_k\|^{-1}\}.
\end{align*}
By making use of a familiar estimate for the divisor function, therefore, we obtain the bound
$$
\Phi(\alp_k,\alp_{k-1})\ll (X/Q)^{k-1+\eps}+X^\eps\sum_{1\le v\le k2^{k^2}(X/Q)^{k-1}}
\min\{ (X/Q)^kv^{-1},\|v\alp_k\|^{-1}\}.$$

\par Suppose that $\alp_k\in \dbR$, and that $c\in \dbZ$ and $r\in \dbN$ satisfy $(c,r)=1$ and 
$|\alp_k-c/r|\le r^{-2}$. Then it follows from \cite[Lemma 2.2]{Vau1997} that
\begin{equation}\label{8.8}
\Phi(\alp_k,\alp_{k-1})\ll (X/Q)^{k+\eps}\( (X/Q)^{-1}+r^{-1}+r(X/Q)^{-k}\).
\end{equation}
Applying a standard transference principle (compare Exercise 2 of \cite[\S2.8]{Vau1997}), it follows that
\begin{equation}\label{8.9}
\Phi(\alp_k,\alp_{k-1})\ll (X/Q)^{k+\eps}\( (X/Q)^{-1}+\grH_{r,c}(\alp)^{-1}+
\grH_{r,c}(\alp)(X/Q)^{-k}\) ,
\end{equation}
where $\grH_{r,c}(\alp)=r+(X/Q)^k|r\alp_k-c|$.\par

We now compare the respective estimates (\ref{8.8}) and (\ref{8.9}) on the one hand, and 
\cite[estimates (10.6) and (10.7)]{Woo2012a} on the other. In this way, one finds that the argument of the
 proof of \cite[Lemma 10.1]{Woo2012a} leading to the estimate (10.10) of that paper may be adapted
 without serious modification to deliver from (\ref{8.6}) and (\ref{8.7}) the bound
\begin{align*}
\Tet_{s,k}(\calQ)&\ll (X/Q)^{-1}\oint |f(\bfalp;2^{k+1}X/Q)|^{2s}\Phi(\alp_k,\alp_{k-1})\d\bfalp \\
&\ll (X/Q)^{k-2+\eps}J_{s,k}(2^{k+1}X/Q)+(X/Q)^{\eps-1}J_{s,k-1}(2^{k+1}X/Q)\\
&\ll (X/Q)^{2s-\frac{1}{2}k(k+1)+\eps}\( (X/Q)^{k-2+\eta(s,k)}+(X/Q)^{k-1+\eta(s,k-1)}\) .
\end{align*}
The conclusion of the lemma now follows.
\end{proof}

In the next phase of our work in this section, we make use of the iterative process from \cite{For1995},
 and this entails the introduction of certain sets of prime numbers. Let $X$ be a large real number
 and for $r\ge 1$ denote by $Y_r$ the set of primes in the interval $(sX^{1/(r(r+1))},2sX^{1/(r(r+1))}]$.
 We adopt the convention in what follows that the empty product is $1$.
 
\begin{lemma}\label{lemma8.3} 
Suppose that $k\ge 3$, $1\le m\le k-1$, $s>m$ and $q=p_1\cdots p_{m-1}$, where each $p_i\in Y_i$.
 Let $\calP_m$ be any set of $2sk^4$ primes in the set $Y_m$. Also, suppose that $b$ is an integer with
 $0\le b<q$ satisfying $(b,q)=1$. Then
$$I_{s,m}(X;q,b)\ll \max_{p\in \calP_m}p^{2s-2m+\frac{3}{2}m(m+1)}\max_{a\in \calB(p)}
I_{s-m,m+1}(X;pq,b+aq),$$
where $\calB(p)=\calB(p;q,b)$ denotes the set of integers $a$ with $0\le a<p$ and $(b+aq,pq)=1$.
\end{lemma}

\begin{proof} This is essentially the special case of \cite[Lemma 4.1]{For1995} in which $f(x)=x^k$.
The statement of \cite[Lemma 4.1]{For1995} has the stronger hypotheses that each $p_i$ be one of the
\emph{smallest} $2sk^4$ primes in $Y_i$, and that $\calP_m$ be the set of $2sk^4$ smallest primes in
 $Y_m$. The argument of the proof, however, shows that the conclusion holds whenever $p_i\in Y_i$ for
 $1\le i\le m-1$ and $\calP_m \subseteq Y_m$.
\end{proof}

\begin{lemma}\label{lemma8.4} When $1\le m\le k-1$, $q\le (2s)^m X^{m/(m+1)}$ and $(b,q)=1$, one
 has
$$I_{s,m}(X;q,b)\ll \Bigl( \prod_{j=m}^{k-2}q^{-1}(X/q)^j\Bigr) I_{s,k-1}(X;q,b).$$
\end{lemma}

\begin{proof} The argument of the proof of \cite[Lemma 4.2]{For1995} shows that for $1\le m\le k-2$,
 one has
$$I_{s,m}(X;q,b)\le \(1+sq^{-1}(X/q)^m \) I_{s,m+1}(X;q,b).$$
The desired conclusion therefore follows by induction on $m$.
\end{proof}

We are now equipped to state and prove the main result of this section. Define the exponential sum
 $g(\alp)=g_k(\alp;X)$ by
$$g_k(\alp;X)=\sum_{1\le x\le X}\er(\alp x^k),$$
and when $s\in \dbN$, define
$$I_s(X)=\int_0^1|g(\alp)|^{2s}\d\alp.$$

\begin{theorem}\label{theorem8.5}
Let $s$ be a natural number. Then whenever $r$ is a natural number with $1\le r\le k-1$, one has
$$I_s(X)\ll X^{2s-k+\eps}\( X^{\eta_r^*(s,k)-1/r}+X^{\eta_r^*(s,k-1)}\),$$
where
$$\eta_r^*(s,w)=r^{-1}\eta(s-\tfrac{1}{2}r(r-1),w).$$
\end{theorem}

\begin{proof} By the Prime Number Theorem, for $1\le i\le r-1$ there is a collection $\calC_i$ of
$\lceil X^{1/(i(i+1))}(2sk^4 \log X)^{-1} \rceil$ disjoint sets of $2sk^4$ primes in the set $Y_i$. Fix
 some choice of sets $\calP_1\in \calC_1$, $\ldots$, $\calP_{r-1}\in \calC_{r-1}$. By applying Lemma
 \ref{lemma8.3}, one finds that whenever $b$ and $q$ satisfy the hypotheses of that lemma, then
$$I_{s-\frac{1}{2}m(m-1),m}(X;q,b)\ll X^{\tfrac{2s}{m(m+1)}+\tfrac{1}{2}}
\max_{p\in \calP_m}\max_{a\in \calB(p)}
I_{s-\frac{1}{2}m(m+1),m+1}(X;pq,b+aq).$$
By iterating this relation, starting with $m=1$ and terminating with Lemma \ref{lemma8.4} at $m=r$, we
 obtain
\begin{equation}\label{8.10}
I_s(X)\ll X^\Ome I_{s-\frac{1}{2}r(r-1),k-1}(X;q,b),
\end{equation}
in which
$$\Ome=2s\sum_{m=1}^{r-1}\frac{1}{m(m+1)}+\frac{r-1}{2}+\sum_{j=r}^{k-2}
\(\frac{j+1}{r}-1\),$$
and $q=p_1\cdots p_{r-1}$ for some prime numbers $p_i\in \calP_i$ $(1\le i\le r-1)$.  A modest
 computation confirms that
\begin{align}
\Ome&=2s(1-1/r)+(r-1)/2+\tfrac{1}{2}k(k-1)/r-\tfrac{1}{2}r(r+1)/r-(k-1-r)\notag \\
&=2s(1-1/r)+\tfrac{1}{2}k(k-1)/r-k+r.\label{8.11}
\end{align}

\par On putting $Q=(2s)^{r-1}X^{1-1/r}$, we see that $2^{-r}Q<q<Q$.  Moreover, distinct choices for
 the $(r-1)$-tuple $\calP_1,\ldots,\calP_{r-1}$ produce distinct numbers $q$.  Therefore, there is a set
 $\calQ$ of integers in the interval $(2^{-r}Q,Q)$ such that \eqref{8.10} holds for each $q\in\calQ$. We 
observe that $(q,k)=1$ for every $q\in \calQ$, and moreover that
\begin{align*}
\text{card}(\calQ) &= \prod_{m=1}^{r-1} \text{card} (\calC_m)\gg \prod_{m=1}^{r-1}
\( X^{1/(m(m+1))}(\log X)^{-1}\) \\
&=X^{1-1/r}(\log X)^{1-r}\gg Q(\log Q)^{1-r}.
\end{align*}
Since $X$ is large, it follows that we may apply Lemma \ref{lemma8.2} to infer that
\begin{align}
\Tet_{s-\frac{1}{2}r(r-1),k}(\calQ)\ll &\,X^\eps (X/Q)^{2s-r(r-1)-\frac{1}{2}(k^2-k+2)}\notag \\
&\,\times \( (X/Q)^{\eta(s-\frac{1}{2}r(r-1),k)-1}+(X/Q)^{\eta(s-\frac{1}{2}r(r-1),k-1)}\)\notag \\
\ll &\,X^\eps (X^{1/r})^{2s-r(r-1)-\frac{1}{2}(k^2-k+2)}
\( X^{\eta_r^*(s,k)-1/r}+X^{\eta_r^*(s,k-1)}\).\label{8.12}
\end{align}

\par Next, on substituting (\ref{8.11}) and (\ref{8.12}) into (\ref{8.10}), we deduce that
$$\sum_{q\in \calQ}I_s(X)\ll X^{2s-k+1-1/r+\eps}\( X^{\eta_r^*(s,k)-1/r}+X^{\eta_r^*(s,k-1)}\).$$
But $\text{card}(\calQ)\gg X^{1-1/r-\eps}$, and so the conclusion of the theorem follows by dividing left
 and right hand side of the last relation by $\text{card}(\calQ)$.
\end{proof}

\section{Application to Waring's problem}

The mean value estimate supplied by our new bounds for $J_{s,k}(X)$ via Theorem \ref{theorem8.5}
 may be utilised to derive improvements in our understanding of the asymptotic formula in Waring's
 problem. Before describing our conclusions, we introduce some notation. We define the set of minor arcs
 $\grm=\grm_k$ to be the set of real numbers $\alp\in [0,1)$ satisfying the property that, whenever 
$a\in \dbZ$ and $q\in \dbN$ satisfy $(a,q)=1$ and $|q\alp-a|\le (2k)^{-1}X^{1-k}$, then
 $q>(2k)^{-1}X$. We recall a mean value estimate restricted to minor arcs.

\begin{theorem}\label{theorem9.1} Suppose that $s\ge k^2-1$. Then for each $\eps>0$, one has
$$\int_\grm |g_k(\alp;X)|^{2s}\d\alp \ll X^{2s-k-1+\eps}.$$
\end{theorem}

\begin{proof} This is \cite[Theorem 10.1]{Woo2013}.
\end{proof}

For each natural number $v$, we define
$$\Del_v^*=\max\{\eta(v,k)-1,\eta(v,k-1)\},$$
where $\eta$ is defined as in the preamble to Lemma \ref{lemma8.1}. Then, for natural numbers $v$ and
 $w$ we put
$$s_0(k,v,w)=2k^2-2-\frac{2k^2-2-(2v+w^2-w)}{1+\Del^*_v/w},$$
and then define
$$s_1(k)=\underset{2v+w^2-w<2k^2-2}{\min_{1\le w\le k-1}\min_{v\ge 1}}s_0(k,v,w).$$

\begin{theorem}\label{theorem9.2}
Suppose that $s$ and $k$ are natural numbers with $k\ge 3$ and $s>s_1(k)$. Then there exists a positive
 number $\del=\del(k,s)$ with the property that
$$\int_\grm |g_k(\alp;X)|^s\d\alp \ll X^{s-k-\del}.$$
\end{theorem}

\begin{proof} The desired conclusion is immediate from Theorem \ref{theorem9.1} in circumstances
 where $s\ge 2k^2-2$, on making use of the trivial estimate $|g_k(\alp;X)|\le X$. We suppose therefore
 that $s_1(k)<s < 2k^2-2$. Let $v$ and 
$w$ be integers with $1\le w\le k-1$, $v\ge 1$ and $2v+w^2-w<2k^2-2$, for which
 $s_1(k)=s_0(k,v,w)$. Then by H\"older's inequality, one has
$$\int_\grm|g(\alp)|^s\d\alp\le \Bigl( \int_\grm |g(\alp)|^{2k^2-2}\d\alp \Bigr)^a
\Bigl( \int_0^1|g(\alp)|^{2v+w^2-w}\d\alp \Bigr)^{1-a},$$
where
$$a=\frac{s-(2v+w^2-w)}{2k^2-2-(2v+w^2-w)}.$$

\par By applying Theorem \ref{theorem9.1} and Theorem \ref{theorem8.5} in sequence, one finds that
\begin{align}
\int_\grm|g(\alp)|^s\d\alp &\ll X^\eps \(X^{2k^2-k-3}\)^a
\(X^{2v+w^2-w-k+\Delta^*_v/w}\)^{1-a}\notag \\
&= X^{s-k-a+(1-a)\Delta^*_v/w+\eps}.\label{9.1}
\end{align}
Since we may suppose that
$$s>s_0(k,v,w)=\frac{(2k^2-2)\Delta^*_v+w(2v+w^2-w)}{w+\Delta^*_v},$$
we see that $a>(1-a)\Delta^*_v/w$, and the conclusion of the theorem follows at once from (\ref{9.1}).
\end{proof}

We now recall some notation associated with the asymptotic formula in Waring's problem. When $s$ and
 $k$ are natural numbers, let $R_{s,k}(n)$ denote the number of representations of the natural number
 $n$ as the sum of $s$ $k$th powers of positive integers. A formal application of the circle method
 suggests that for $k\ge 3$ and $s\ge k+1$, one should have
\begin{equation}\label{9.2}
R_{s,k}(n)=\frac{\Gam(1+1/k)^s}{\Gam(s/k)}\grS_{s,k}(n)n^{s/k-1}+o(n^{s/k-1}),
\end{equation}
where
$$\grS_{s,k}(n)=\sum_{q=1}^\infty \sum^q_{\substack{a=1\\ (a,q)=1}}\( q^{-1}\sum_{r=1}^q
 \er(ar^k/q)\)^s \er(-na/q).$$
Subject to suitable congruence conditions, one has $1\ll \grS_{s,k}(n)\ll n^\eps$, so that the conjectured
 relation (\ref{9.2}) represents an honest asymptotic formula. Let $\Gtil(k)$ denote the least integer $t$
 with the property that, for all $s\ge t$, and all sufficiently large natural numbers $n$, one has the
 asymptotic formula (\ref{9.2}).\par

The argument following the proof of \cite[Lemma 3.1]{Woo2012b} may be adapted in the present
 circumstances to show that $\Gtil(k)\le [s_1(k)]+1$ for $k\ge 3$. For each natural number 
$m\le \frac{1}{2}k$, we find from Theorem \ref{theorem1.2} that when $v=(k-m)^2+(k-m)$, one has
$$\eta(v,k)-1\le m^2-1\quad \text{and}\quad \eta(v,k-1)\le (m-1)^2,$$
so that
\begin{equation}\label{9.3}
\Del_v^* \le m^2-1\quad \text{for}\quad v=(k-m)^2+(k-m).
\end{equation}
Similarly, again from Theorem \ref{theorem1.2}, for each natural number $m\le \tfrac{1}{2}(k-1)$, we find 
that when $v=(k-m)^2-1$, one has
$$\eta(v,k)-1\le m^2+m-1+\frac{m}{k-m-1}$$
and
$$\eta(v,k-1)\le (m-1)^2+(m-1)+\frac{m-1}{k-m},$$
so that
\begin{equation}\label{9.4}
\Del_v^* \le m^2+m-1+\frac{m}{k-m-1}\quad \text{for}\quad v=(k-m)^2-1.
\end{equation}
Employing these exponents (\ref{9.3}) and (\ref{9.4}) in order to obtain upper bounds for $s_1(k)$, we
 obtain the upper bounds for $\Gtil(k)$ recorded in the following corollary.

\begin{corollary}\label{corollary9.3}
One has
$$\Gtil(12)\le 253,\quad \Gtil(13)\le 299,\quad \Gtil(14)\le 349,\quad \Gtil(15)\le 403,\quad 
\Gtil(16)\le 460,$$
$$\Gtil(17)\le 521,\quad \Gtil(18)\le 587,\quad \Gtil(19)\le 656,\quad \Gtil(20)\le 729.$$
\end{corollary}

We note that in each of these bounds, it is (\ref{9.4}) which is utilised within the formula for $s_1(k)$.
 One takes $m=2$ for $k=12$, and $m=3$ for $13\le k\le 20$. Meanwhile, one takes $w=5$ for $k=12$,
 $w=6$ for $k=13,14$, and $w=7$ for $15\le k\le 20$.\par

For comparison, the bounds for $\Gtil(k)$ made available in \cite[Corollary 1.7]{Woo2013} show that
$$\Gtil(12)\le 255,\quad \Gtil(13)\le 303,\quad \Gtil(14)\le 354,\quad \Gtil(15)\le 410,\quad 
\Gtil(16)\le 470,$$
$$\Gtil(17)\le 534,\quad \Gtil(18)\le 602,\quad \Gtil(19)\le 674,\quad \Gtil(20)\le 748.$$
For $k\le 11$, the bounds for $\Gtil(k)$ in \cite[Corollary 1.7]{Woo2013} prove superior to those that
 follow from the work of this paper. For large values of $k$, meanwhile, the conclusion of 
\cite[Corollary 1.6]{Woo2013} shows that
$$\Gtil(k)\le 2k^2-k^{4/3}+O(k).$$
We are able to provide a modest improvement in this bound as a consequence of 
Theorem \ref{theorem9.2}.

\begin{corollary}\label{corollary9.4}
When $k$ is a large natural number, one has
$$\Gtil(k)\le 2k^2-2^{2/3}k^{4/3}+O(k).$$
\end{corollary}

\begin{proof}As we have already noted, one has $\Gtil(k)\le [s_1(k)]+1$, and so it suffices to bound
 $s_1(k)$ for large values of $k$. We take
$$m=[2^{2/3}k^{1/3}],\quad v=(k-m)^2+(k-m)\quad \text{and}\quad w=[2^{1/3}k^{2/3}],$$
 so that from (\ref{9.3}) one obtains
\begin{align*}
s_0(k,v,w)&\le 2k^2-2-\frac{2k^2-2-2(k^2-2mk)-w^2+O(k)}{1+m^2/w+O(k^{-2/3})}\\
&=2k^2-2-\frac{4(2^{2/3}k^{1/3})k-2^{2/3}k^{4/3}+O(k)}{3+O(k^{-1/3})}\\
&=2k^2-2^{2/3}k^{4/3}+O(k).
\end{align*}
This confirms the conclusion of the corollary.
\end{proof}

\bibliographystyle{amsbracket}
\providecommand{\bysame}{\leavevmode\hbox to3em{\hrulefill}\thinspace}

\end{document}